\documentclass[11pt, twoside]{article}

\usepackage{amsmath}
\usepackage{amssymb}
\usepackage{titletoc}
\usepackage{mathrsfs}
\usepackage{amsthm}
\usepackage{color}
\usepackage{latexsym,amssymb}

\usepackage{indentfirst}
\usepackage{color}
\usepackage{txfonts}

\allowdisplaybreaks

\def\bint{{\ifinner\rlap{\bf\kern.30em--}
\int\else\rlap{\bf\kern.35em--}\int\fi}\ignorespaces}

\def\sbint{{\ifinner\rlap{\bf\kern.32em--}
\hspace{0.078cm}\int\else\rlap{\bf\kern.45em--}\int\fi}\ignorespaces}

\def\red{\color{red}}

\def\rr{{\mathbb R}}
\def\rn{{\mathbb{R}^n}}

\def\nn{{\mathbb N}}
\def\zz{{\mathbb Z}}

\def\fz{\infty }

\def\dz{\delta}

\def\lz{\lambda}

\def\lf{\left}
\def\r{\right}
\def\ls{\lesssim}
\def\noz{\nonumber}
\def\wz{\widetilde}

\def\loc{{\mathrm{loc}}}
\DeclareMathOperator{\supp}{supp}

\def\XXint#1#2#3{{\setbox0=\hbox{$#1{#2#3}{\int}$ }
\vcenter{\hbox{$#2#3$ }}\kern-.6\wd0}}

\def\f{\frac}

\def\lz{{\lambda}}

\newtheorem{theorem}{Theorem}[section]
\newtheorem{lemma}[theorem]{Lemma}

\newtheorem{proposition}[theorem]{Proposition}

\newtheorem{assumption}[theorem]{Assumption}
\theoremstyle{definition}
\newtheorem{remark}[theorem]{Remark}
\newtheorem{definition}[theorem]{Definition}
\renewcommand{\appendix}{\par
\setcounter{section}{0}%
\setcounter{subsection}{0}%
\setcounter{subsubsection}{0}%
\gdef\thesection{\@Alph\c@section}%
\gdef\thesubsection{\@Alph\c@section.\@arabic\c@subsection}%
\gdef\theHsection{\@Alph\c@section.}%
\gdef\theHsubsection{\@Alph\c@section.\@arabic\c@subsection}%
\csname appendixmore\endcsname
}

\numberwithin{equation}{section}

\begin{document}

\arraycolsep=1pt

\title{\bf\Large
Boundedness of Calder\'on--Zygmund operators on
ball Campanato-type function spaces
\footnotetext{\hspace{-0.35cm} 2020 {\it
Mathematics Subject Classification}. Primary 42B20;
Secondary 42B25, 42B30, 42B35, 46E35, 47A30.
\endgraf {\it Key words and phrases.}
Calder\'on--Zygmund operator, Campanato-type space,
ball quasi-Banach function space, Hardy-type space.
\endgraf This project is partially supported by
the National Natural Science Foundation of China (Grant Nos.\
11971058 and 12071197) and the National
Key Research and Development Program of China
(Grant No.\ 2020YFA0712900).}}
\date{}
\author{Yiqun Chen, Hongchao Jia and Dachun Yang\footnote{Corresponding author,
E-mail: \texttt{dcyang@bnu.edu.cn}/{\red August 12, 2022}/Final version.}}
\maketitle

\vspace{-0.8cm}

\begin{center}
\begin{minipage}{13cm}
{\small {\bf Abstract}\quad
Let $X$ be a ball quasi-Banach function space on ${\mathbb R}^n$
satisfying some mild assumptions.
In this article, the authors first find a reasonable version
$\widetilde{T}$ of the Calder\'on--Zygmund operator $T$
on the ball Campanato-type function space $\mathcal{L}_{X,q,s,d}(\mathbb{R}^n)$
with $q\in[1,\infty)$, $s\in\mathbb{Z}_+^n$, and $d\in(0,\infty)$.
Then the authors prove that
$\widetilde{T}$ is bounded on $\mathcal{L}_{X,q,s,d}(\mathbb{R}^n)$
if and only if, for any $\gamma\in\mathbb{Z}^n_+$
with $|\gamma|\leq s$, $T^*(x^{\gamma})=0$, which is hence sharp.
Moreover, $\widetilde{T}$ is proved
to be the adjoint operator of $T$, which further strengthens
the rationality of the definition of $\widetilde{T}$.
All these results have a wide range of applications.
In particular, even when they are
applied, respectively, to weighted Lebesgue spaces, variable Lebesgue spaces,
Orlicz spaces, Orlicz-slice spaces,
Morrey spaces, mixed-norm Lebesgue spaces,
local generalized Herz spaces, and mixed-norm Herz spaces, all
the obtained results are new. The proofs of these results strongly
depend on the properties of the kernel of $T$ under consideration
and also on the dual theorem on $\mathcal{L}_{X,q,s,d}(\mathbb{R}^n)$.
}
\end{minipage}
\end{center}

\vspace{0.2cm}



\section{Introduction\label{s-intro}}

Recall that, for any $q\in(0,\infty]$, the \emph{Lebesgue space} $L^q(\rn)$
is defined to be the set of all the measurable functions $f$ on $\rn$
such that
$$\|f\|_{L^q(\rn)}:=
\begin{cases}
\displaystyle
\lf[\int_{\rn}|f(x)|^q\, dx\r]^{\frac{1}{q}}
&\text{if}\quad q\in(0,\fz),\\
\displaystyle
\mathop{\mathrm{ess\,sup}}_{x\in\rn}\,|f(x)|
&\text{if}\quad q=\fz
\end{cases}$$
is finite. It is well known that Lebesgue spaces play a fundamental role
in the modern analysis of mathematics.
Recall that the Calder\'on--Zygmund operator $T$ is bounded on $L^p(\rn)$ with $p\in(1,\fz)$,
which no longer holds true for $L^\fz(\rn)$.
Indeed, $T$ maps $L^{\fz}(\rn)$ to the space $\mathop{\mathrm{BMO}\,}(\rn)$
of functions with bounded mean oscillation, which was introduced by
John and Nirenberg \cite{JN}. This implies that the space $\mathop{\mathrm{BMO}\,}(\rn)$
is a reasonable replacement of $L^{\fz}(\rn)$ when
studying the boundedness of Calder\'on--Zygmund operators.
Moreover, in recent years,
the boundedness of Calder\'on--Zygmund operators on
BMO-type spaces has attracted more and more attention.
For instance, Nakai \cite{N10, N17} obtained the boundedness
of Calder\'on--Zygmund operators
on both Campanato-type spaces with variable growth conditions
over spaces of homogeneous type, in the sense of Coifman and Weiss, and their predual spaces.
We refer the reader also to \cite{an2019,HSS2016,ho2019,NS2012,yn2021,yn2022}
for more studies on this.

Recall that Sawano et al. \cite{SHYY}
introduced the ball quasi-Banach function space $X$ and
the related Hardy space $H_X(\rn)$ which is a natural generalization
of the classical Hardy space $H^p(\rn)$.
Ball quasi-Banach function spaces contain many well-known function spaces, for
instance, the weighted Lebesgue space, the Morrey space, the mixed-norm Lebesgue space,
and the Orlicz--slice space, but they may not be
quasi-Banach function spaces (see \cite{SHYY,zyyw,zhyy2022} for the details).
This implies that the studies on $X$ (see, for instance,
\cite{CWYZ2020,dgpyyz,ho2021,is2017,SHYY,tyyz,wyy,wyyz,zyyw}) have more wide applications.
We also refer the reader to \cite{syy,syy2,yhyy,yhyy1} for studies on
(weak) Hardy spaces associated with $X$ on spaces of homogeneous type in
the sense of Coifman and Weiss.

On the other hand, Fefferman and Stein \cite{FS72} obtained the well-known dual theorem
that $(H^1(\rn))^*=\mathop{\mathrm{BMO}\,}(\rn)$. Later,
Taibleson and Weiss \cite{tw80} generalized the above dual
theorem to the classical Hardy space $H^p(\rn)$
for any given $p\in(0,1]$, namely, for any $q\in[1,\infty]$
and any integer $s\in[0,\infty)\cap[n(\frac{1}{p}-1),\infty)$,
the dual space of $H^p(\rn)$ is $\mathcal{C}_{\frac{1}{p}-1,q,s}(\rn)$,
the Campanato space introduced in \cite{C}, which coincides with
$\mathop{\mathrm{BMO}\,}(\rn)$ when $p=1$.
Very recently, motivated by
$(H^p(\rn))^*=\mathcal{C}_{\frac{1}{p}-1,q,s}(\rn)$,
Zhang et al. \cite{zhyy2022} studied the dual theorem
of $H_X(\rn)$. To be precise, Zhang et al. \cite{zhyy2022}
proved that the dual space of $H_X(\rn)$ is $\mathcal{L}_{X,q,s,d}(\rn)$,
the ball Campanato-type function space associated with $X$ and, moreover,
established the Carleson measure characterization of $\mathcal{L}_{X,q,s,d}(\rn)$.

In this article, we study the boundedness of
Calder\'on--Zygmund operators on $\mathcal{L}_{X,q,s,d}(\rn)$.
To be precise, let $X$ be a ball quasi-Banach function space on ${\mathbb R}^n$
satisfying some mild assumptions.
We first find a reasonable version
$\widetilde{T}$ of the Calder\'on--Zygmund operator $T$
on the ball Campanato-type function space $\mathcal{L}_{X,q,s,d}(\mathbb{R}^n)$
(see Definition \ref{def-B-CZO} below), which was originally introduced
in \cite{ms1979} in the case $s:=0$ and further studied by Nakai \cite{N10,N17}
on spaces of homogeneous type.
Then we prove that
$\widetilde{T}$ is bounded on $\mathcal{L}_{X,q,s,d}(\mathbb{R}^n)$
if and only if, for any $\gamma\in\mathbb{Z}^n_+$
with $|\gamma|\leq s$, $T^*(x^{\gamma})=0$, which is hence sharp.
Moreover, $\widetilde{T}$ is proved
to be the adjoint operator of $T$, which further strengthens
the rationality of the definition of $\widetilde{T}$.
All these results have a wide range of applications.
In particular, even when they are
applied, respectively, to weighted Lebesgue spaces, variable Lebesgue spaces,
Orlicz spaces, Orlicz-slice spaces,
Morrey spaces, mixed-norm Lebesgue spaces,
local generalized Herz spaces, and mixed-norm Herz spaces, all
the obtained results are new. Due to the generality and the
practicability, more applications of these main results of this
article are predictable. The proofs of these results strongly
depend on the properties of the kernel of $T$ under consideration
and also on the dual theorem that $(H_X(\rn))^*=\mathcal{L}_{X,q,s,d}(\mathbb{R}^n)$.

To be precise, the remainder of this article is organized as follows.

In Section \ref{s1}, we recall some basic concepts
on the quasi-Banach function space $X$ (see Definition \ref{Debqfs} below)
and its related ball Campanato-type function space (see Definition \ref{2d2} below).

Section \ref{sec-CZO-B} is divided into two parts.
In Subsection \ref{sec-CZO-B-01}, we find a reasonable version
$\widetilde{T}$ of the Calder\'on--Zygmund operator $T$
on $\mathcal{L}_{X,q,s,d}(\mathbb{R}^n)$.
Then we establish the boundedness of $\widetilde{T}$ on $\mathcal{L}_{X,q,s,d}(\mathbb{R}^n)$
(see Theorem \ref{thm-B-CZ} below).
In Subsection \ref{sec-CZO-B-02}, we prove that $\widetilde{T}$
is the adjoint operator of $T$ (see Theorem \ref{dual-T-wT} below).

In Section \ref{Appli}, we apply all the above main results to eight concrete examples
of ball quasi-Banach function spaces,
namely, the weighted Lebesgue space
$L^p_{w}(\rn)$, the variable Lebesgue space $L^{p(\cdot)}(\rn)$,
the Orlicz space $L^{\Phi}(\rn)$, the Orlicz-slice space $(E_\Phi^r)_t(\rn)$,
the Morrey space $M_r^p(\rn)$,
the mixed-norm Lebesgue space $L^{\vec{p}}(\rn)$,
the local generalized Herz space
$\dot{\mathcal{K}}_{\omega,\mathbf{0}}^{p,r}(\mathbb{R}^{n})$,
and the mixed-norm Herz space $\dot{E}^{\vec{\alpha},\vec{p}}_{\vec{q}}(\rn)$,
respectively. Therefore, all the boundedness of $\wz T$
on $\mathcal{L}_{L^p_{w}(\rn),q,s,d}(\rn)$, $\mathcal{L}_{L^{p(\cdot)}(\rn),q,s,d}(\rn)$,
$\mathcal{L}_{L^{\Phi}(\rn),q,s,d}(\rn)$, and $\mathcal{L}_{(E_\Phi^r)_t(\rn),q,s,d}(\rn)$,
$\mathcal{L}_{M_r^p(\rn),q,s,d}(\rn)$, $\mathcal{L}_{L^{\vec{p}}(\rn),q,s,d}(\rn)$,
$\mathcal{L}_{\dot{\mathcal{K}}_{\omega,\mathbf{0}}^{p,r}(\mathbb{R}^{n}),q,s,d}(\rn)$,
and $\mathcal{L}_{\dot{E}^{\vec{\alpha},\vec{p}}_{\vec{q}}(\rn),q,s,d}(\rn)$ are obtained
(see, respectively, Theorems \ref{thm-wei},
\ref{thm-vec}, \ref{thm-phi}, \ref{thm-Et}, \ref{ap-M},
\ref{apply3}, \ref{apply5}, and \ref{apply7} below).

Finally, we make some conventions on notation. Let
$\nn:=\{1,2,\ldots\}$, $\zz_+:=\nn\cup\{0\}$, and $\mathbb{Z}^n_+:=(\mathbb{Z}_+)^n$.
We always denote by $C$ a \emph{positive constant}
which is independent of the main parameters,
but it may vary from line to line.
The symbol $f\lesssim g$ means that $f\le Cg$.
If $f\lesssim g$ and $g\lesssim f$, we then write $f\sim g$.
If $f\le Cg$ and $g=h$ or $g\le h$,
we then write $f\ls g\sim h$
or $f\ls g\ls h$, \emph{rather than} $f\ls g=h$
or $f\ls g\le h$. For any $s\in\zz_+$, we use $\mathcal{P}_s(\rn)$
to denote the set of all the polynomials on $\rn$
with total degree not greater than $s$. We use $\mathbf{0}$ to
denote the \emph{origin} of $\rn$.
For any measurable subset $E$ of $\rn$, we denote by $\mathbf{1}_E$ its
characteristic function. Moreover, for any $x\in\rn$ and $r\in(0,\fz)$,
let $B(x,r):=\{y\in\rn:\ |y-x|<r\}$.
Furthermore, for any $\lambda\in(0,\infty)$
and any ball $B(x,r)\subset\rn$ with $x\in\rn$ and
$r\in(0,\fz)$, let $\lambda B(x,r):=B(x,\lambda r)$.
Finally, for any $q\in[1,\infty]$,
we denote by $q'$ its \emph{conjugate exponent},
namely, $\frac{1}{q}+\frac{1}{q'}=1$.
Also, when we prove a lemma, proposition,
theorem, or corollary, we always use the same symbols
in the wanted proved lemma, proposition, theorem, or corollary.

\section{Ball quasi-Banach function spaces\label{s1}}

In this section, we recall the definitions of ball quasi-Banach function spaces and their related
ball Campanato-type function spaces.
In what follows, let $\mathscr M(\rn)$ denote the set of
all measurable functions on $\rn$ and
\begin{equation}\label{Eqball}
\mathbb{B}(\rn):=\lf\{B(x,r):\ x\in\rn \text{ and } r\in(0,\infty)\r\}.
\end{equation}
The following definition of ball quasi-Banach function
spaces on $\rn$
is just \cite[Definition 2.2]{SHYY}.

\begin{definition}\label{Debqfs}
Let $X\subset\mathscr{M}(\rn)$ be a quasi-normed linear space
equipped with a quasi-norm $\|\cdot\|_X$ which makes sense for all
measurable functions on $\rn$.
Then $X$ is called a \emph{ball quasi-Banach
function space} on $\rn$ if it satisfies:
\begin{enumerate}
\item[$\mathrm{(i)}$] if $f\in\mathscr{M}(\rn)$, then $\|f\|_{X}=0$ implies
that $f=0$ almost everywhere;
\item[$\mathrm{(ii)}$] if $f, g\in\mathscr{M}(\rn)$, then $|g|\le |f|$ almost
everywhere implies that $\|g\|_X\le\|f\|_X$;
\item[$\mathrm{(iii)}$] if $\{f_m\}_{m\in\nn}\subset\mathscr{M}(\rn)$ and $f\in\mathscr{M}(\rn)$,
then $0\le f_m\uparrow f$ almost everywhere as $m\to\infty$
implies that $\|f_m\|_X\uparrow\|f\|_X$ as $m\to\infty$;
\item[$\mathrm{(iv)}$] $B\in\mathbb{B}(\rn)$ implies
that $\mathbf{1}_B\in X$,
where $\mathbb{B}(\rn)$ is the same as in \eqref{Eqball}.
\end{enumerate}

Moreover, a ball quasi-Banach function space $X$
is called a
\emph{ball Banach function space} if it satisfies:
\begin{enumerate}
\item[$\mathrm{(v)}$] for any $f,g\in X$,
\begin{equation*}
\|f+g\|_X\le \|f\|_X+\|g\|_X;
\end{equation*}
\item[$\mathrm{(vi)}$] for any ball $B\in \mathbb{B}(\rn)$,
there exists a positive constant $C_{(B)}$,
depending on $B$, such that, for any $f\in X$,
\begin{equation*}
\int_B|f(x)|\,dx\le C_{(B)}\|f\|_X.
\end{equation*}
\end{enumerate}
\end{definition}

\begin{remark}\label{rem-ball-B}
\begin{enumerate}
\item[$\mathrm{(i)}$] Let $X$ be a ball quasi-Banach
function space on $\rn$. By \cite[Remark 2.6(i)]{yhyy1},
we conclude that, for any $f\in\mathscr{M}(\rn)$, $\|f\|_{X}=0$ if and only if $f=0$
almost everywhere.

\item[$\mathrm{(ii)}$] As was mentioned in
\cite[Remark 2.6(ii)]{yhyy1}, we obtain an
equivalent formulation of Definition \ref{Debqfs}
via replacing any ball $B$ by any
bounded measurable set $E$ therein.

\item[$\mathrm{(iii)}$] We should point out that,
in Definition \ref{Debqfs}, if we
replace any ball $B$ by any measurable set $E$ with
finite measure, we obtain the
definition of (quasi-)Banach function spaces which were originally
introduced in \cite[Definitions 1.1 and 1.3]{BS88}. Thus,
a (quasi-)Banach function space
is also a ball (quasi-)Banach function
space and the converse is not necessary to be true.
\item[$\mathrm{(iv)}$] By \cite[Theorem 2]{dfmn2021},
we conclude that both (ii) and (iii) of
Definition \ref{Debqfs} imply that any ball quasi-Banach
function space is complete.
\end{enumerate}
\end{remark}

We also recall the concepts of both the convexity and
the concavity of ball quasi-Banach function spaces,
which are a part of \cite[Definition 2.6]{SHYY}.

\begin{definition}\label{Debf}
	Let $X$ be a ball quasi-Banach function space and $p\in(0,\infty)$.
	\begin{enumerate}
		\item[(i)] The \emph{$p$-convexification} $X^p$ of $X$
		is defined by setting
		$$X^p:=\lf\{f\in\mathscr M(\rn):\ |f|^p\in X\r\}$$
		equipped with the \emph{quasi-norm} $\|f\|_{X^p}:=\|\,|f|^p\,\|_X^{1/p}$ for any $f\in X^p$.
		
		\item[(ii)] The space $X$ is said to be
		\emph{concave} if there exists a positive constant
		$C$ such that, for any $\{f_k\}_{k\in{\mathbb N}}\subset \mathscr M(\rn)$,
		$$\sum_{k=1}^{{\infty}}\|f_k\|_{X}
		\le C\left\|\sum_{k=1}^{{\infty}}|f_k|\right\|_{X}.$$
		In particular, when $C=1$, $X$ is said to be
		\emph{strictly concave}.
	\end{enumerate}
\end{definition}

Next, we present the definition of ball Campanato-type function spaces
associated with $X$ introduced in \cite[Definition 3.2]{zhyy2022}.
In what follows, for any given $q\in(0,\infty]$, we use
$L^q_{\mathrm{loc}}(\rn)$ to denote
the set of all the measurable functions $f$ on $\rn$ such that
$f\mathbf{1}_E\in L^q(\rn)$
for any bounded measurable set $E\subset \rn$.
For any $f\in L_{\loc}^1(\rn)$
and any finite measurable subset $E\subset\rn$, let
$$f_E:=\fint_Ef(x)\,dx:=\frac{1}{|E|}\int_E f(x)\,dx.$$
Moreover, for any $s\in\zz_+$, $\mathcal{P}_s(\rn)$
denotes the set of all the polynomials on $\rn$
with total degree not greater than $s$; for any ball
$B\in\mathbb{B}(\rn)$ and
any locally integrable function $g$ on $\rn$,
$P^{(s)}_B(g)$ denotes the \emph{minimizing polynomial} of
$g$ with total degree not greater than $s$, namely,
$P^{(s)}_B(g)$ is the unique polynomial $f\in\mathcal{P}_s(\rn)$
such that, for any $P\in\mathcal{P}_s(\rn)$,
\begin{align*}
\int_{B}[g(x)-f(x)]P(x)\,dx=0.
\end{align*}
Recall that Yan et al. \cite[Definition 1.11]{yyy20} introduced the following
Campanato space $\mathcal{L}_{X,q,s}(\rn)$
associated with the ball quasi-Banach function space $X$.

\begin{definition}\label{cqb}
Let $X$ be a ball quasi-Banach function space, $q\in[1,\fz)$,
and $s\in\zz_+$. The \emph{Campanato space}
$\mathcal{L}_{X,q,s}(\rn)$, associated with $X$, is defined to be
the set of all the $f\in L^q_{\mathrm{loc}}(\rn)$ such that
$$\|f\|_{\mathcal{L}_{X,q,s}(\rn)}
:=\sup_{B\subset\rn}\frac{|B|}{\|\mathbf{1}_B\|_X}\lf\{
\fint_{B}\lf|f(x)-P^{(s)}_B(f)(x)\r|^q\,dx\r\}^{1/q}<\infty,$$
where the supremum is taken over all balls $B\in\mathbb{B}(\rn)$.
\end{definition}

In what follows, by abuse of notation, we identify
$f\in\mathcal{L}_{X,q,s}(\rn)$ with $f+\mathcal{P}_s(\rn)$.

\begin{remark}\label{comeback}
Let $q\in[1,\infty)$, $s\in\zz_+$, and $\alpha\in[0,\fz)$.
Recall that Campanato \cite{C} introduced the
following well-known \emph{Campanato space}
$\mathcal{C}_{\alpha,q,s}(\rn)$ which is defined to be the
set of all the $f\in L^q_\loc(\rn)$ such that
\begin{align*}
\|f\|_{\mathcal{C}_{\alpha,q,s}(\rn)} \
:=\sup |B|^{-\alpha}\left[\fint_{B}\left|f(x)
-P^{(s)}_B(f)(x)\right|^{q}\,dx\right]^{\frac{1}{q}}<\infty,
\end{align*}
where the supremum is taken over all balls $B\subset\rn$.
It is well known that $\mathcal{C}_{\alpha,q,s}(\rn)$ when $\alpha=0$
coincides with the space $\mathrm{BMO}\,(\rn)$.
Moreover, when $X:=L^{\frac{1}{\alpha+1}}(\rn)$, we have
$\mathcal{L}_{X,q,s}(\rn)=\mathcal{C}_{\alpha,q,s}(\rn)$.
\end{remark}

Very recently, Zhang et al. \cite[Definition 3.2]{zhyy2022}
introduced the following
ball Campanato-type function space $\mathcal{L}_{X,q,s,d}(\rn)$ associated with the
ball quasi-Banach function space $X$.

\begin{definition}\label{2d2}
Let $X$ be a ball quasi-Banach function space, $q\in[1,{\infty})$,
$d\in(0,\infty)$,
and $s\in\zz_+$. Then the \emph{ball Campanato-type function space}
$\mathcal{L}_{X,q,s,d}(\rn)$, associated with $X$, is defined to be
the set of all the $f\in L^q_{\mathrm{loc}}({{\rr}^n})$ such that
\begin{align*}
\|f\|_{\mathcal{L}_{X,q,s,d}(\rn)}
:&=\sup
\lf\|\lf\{\sum_{i=1}^m
\lf(\frac{{\lambda}_i}{\|{\mathbf{1}}_{B_i}\|_X}\r)^d
{\mathbf{1}}_{B_i}\r\}^{\frac1d}\r\|_{X}^{-1}\\
&\quad\times\sum_{j=1}^m\frac{{\lambda}_j|B_j|}{\|{\mathbf{1}}_{B_j}
\|_{X}}
\lf[\fint_{B_j}\lf|f(x)-P^{(s)}_{B_j}(f)(x)\r|^q \,dx\r]^\frac1q
\end{align*}
is finite, where the supremum is taken over all
$m\in\nn$, $\{B_j\}_{j=1}^m\subset \mathbb{B}(\rn)$, and
$\{\lambda_j\}_{j=1}^m\subset[0,\infty)$ with
$\sum_{j=1}^m\lambda_j\neq0$.
\end{definition}

In what follows, by abuse of notation, we identify
$f\in\mathcal{L}_{X,q,s,d}(\rn)$ with $f+\mathcal{P}_s(\rn)$.

\begin{remark}\label{rem-ball-B3}
Let $q\in[1,\fz)$, $s\in\zz_+$, $d\in(0,1)$, and
$X$ be a concave ball quasi-Banach function space.
In this case, it was proved in
\cite[Proposition 3.7]{zhyy2022} that
$\mathcal{L}_{X,q,s,d}({{\rr}^n})=\mathcal{L}_{X,q,s}(\rn)$
with equivalent quasi-norms.
\end{remark}

\section{Calder\'on--Zygmund operators on
$\mathcal{L}_{X,q,s,d}(\rn)$}\label{sec-CZO-B}

This section is devoted to studying the boundedness
of Calder\'on--Zygmund operators on the space $\mathcal{L}_{X,q,s,d}(\rn)$.
For this purpose,
we divide this section into two subsections.
In Subsection \ref{sec-CZO-B-01}, we present a reasonable version
$\widetilde{T}$ of the Calder\'on--Zygmund operator $T$
on $\mathcal{L}_{X,q,s,d}(\mathbb{R}^n)$ and then
obtain the boundedness of
$\widetilde{T}$ on $\mathcal{L}_{X,q,s,d}(\mathbb{R}^n)$.
In Subsection \ref{sec-CZO-B-02}, we prove that $\widetilde{T}$ is
the adjoint operator of $T$, which further strengthens
the rationality of the definition of $\widetilde{T}$.

\subsection{Boundedness of Calder\'on--Zygmund operators on
$\mathcal{L}_{X,q,s,d}(\rn)$}\label{sec-CZO-B-01}

In this subsection, we establish the boundedness
of Calder\'on--Zygmund operators on the space $\mathcal{L}_{X,q,s,d}(\rn)$.
We begin with the concept of the $s$-order standard kernel
(see, for instance, \cite[Chapter III]{EMS2}).
In what follows, for any $\gamma=(\gamma_1,\ldots,\gamma_n)\in
\zz_+^n$,
any $\gamma$-order differentiable function $F(\cdot,\cdot)$
on $\rn\times \rn\setminus\{(w,w):\ w\in\rn\}$, and any $(x,y)\in \rn\times \rn\setminus\{(w,w):\ w\in\rn\}$, let
$$
\partial_{(1)}^{\gamma}F(x,y):=\frac{\partial^{|\gamma|}}
{\partial x_1^{\gamma_1}\cdots\partial x_n^{\gamma_n}}F(x,y)$$
and
$$\partial_{(2)}^{\gamma}F(x,y):=\frac{\partial^{|\gamma|}}
{\partial y_1^{\gamma_1}\cdots\partial y_n^{\gamma_n}}F(x,y).
$$

\begin{definition}\label{def-s-k}
Let $s\in\zz_+$ and $\dz\in(0,1]$. A measurable function $K$
on $\rn\times \rn\setminus\{(w,w):\ w\in\rn\}$
is called an \emph{$s$-order standard kernel with regularity $\dz$} if
there exists a positive constant $C$
such that, for any $\gamma\in\zz_+^n$ with $|\gamma|\leq s$,
the following hold true:
\begin{itemize}
\item[\rm (i)]
for any $x,y\in\rn$ with $x\neq y$,
\begin{align}\label{size-s'}
\lf|\partial_{(2)}^{\gamma}K(x,y)\r|\le
\frac{C}{|x-y|^{n+|\gamma|}};
\end{align}

\item[\rm (ii)]
\eqref{size-s'} still holds true for the first variable of $K$;

\item[\rm (iii)]
for any $x,y,z\in\rn$
with $x\neq y$ and $|x-y|\ge2|y-z|$,
\begin{align}\label{regular2-s}
\lf|\partial_{(2)}^{\gamma}K(x,y)-\partial_{(2)}^{\gamma}K(x,z)\r|
\le C\frac{|y-z|^\dz}{|x-y|^{n+|\gamma|+\dz}};
\end{align}

\item[\rm (iv)]
\eqref{regular2-s} still holds true for the first variable of $K$.
\end{itemize}
\end{definition}

In what follows, we use $\eta\to 0^+$
to denote $\eta \in(0,\fz)$ and $\eta\to 0$.

\begin{definition}\label{defin-C-Z-s}
Let $s\in\zz_+$ and $K$ be the same as in Definition \ref{def-s-k}.
A linear operator $T$ is called an
\emph{$s$-order Calder\'on--Zygmund singular integral operator
with kernel $K$} if $T$ is bounded on $L^2(\rn)$ and,
for any given $f\in L^2(\rn)$
and for almost every $x\in\rn$,
\begin{align*}
T(f)(x)=\lim_{\eta\to0^+}T_\eta (f)(x),
\end{align*}
where, for any $\eta\in(0,\infty)$,
\begin{align*}
T_\eta (f)(x)
:=\int_{\rn\setminus B(x,\eta)}K(x,y)f(y)\,dy.
\end{align*}
\end{definition}

\begin{remark}\label{rem-2.15}
Let $T$ be the same as in Definition \ref{defin-C-Z-s}.
From a claim in \cite[p.\,102]{Duo01}, we deduce that, for any $q\in(1,\fz)$,
$T$ is bounded on $L^q(\rn)$ and, moreover, for any $f\in L^q(\rn)$,
\begin{align*}
T(f)=\lim_{\eta\to0^+}T_\eta (f)
\end{align*}
both almost everywhere on $\rn$ and in $L^q(\rn)$.
\end{remark}

Now, we recall the following well-known
concept of vanishing moments on $T$
(see, for instance, \cite[p.\,23]{mc1997}).

\begin{definition}\label{Def-T-s-v}
Let $s\in\zz_+$. An $s$-order
Calder\'on--Zygmund singular integral operator $T$
is said to have the \emph{vanishing moments up to
order $s$} if, for any function $a\in L^2(\rn)$ having compact support
and satisfying that,
for any $\gamma:=(\gamma_1,\ldots,\gamma_n)\in\zz_+^n$ with $|\gamma|\le s$,
$\int_{\rn} a(x)x^\gamma\,dx=0$, it holds true that
\begin{align*}
T^*(x^{\gamma}):=\int_{\rn} T(a)(x)x^\gamma\,dx=0,
\end{align*}
here and thereafter, $|\gamma|:=\gamma_1+\cdots+\gamma_n$ and,
for any $x:=(x_1,\ldots,x_n)\in\rn$,
$x^\gamma:=x_1^{\gamma_1}\cdots x_n^{\gamma_n}$.
\end{definition}

To show the boundedness of Calder\'on--Zygmund operators on both
$\mathcal{L}_{X,q,s}(\rn)$ and $\mathcal{L}_{X,q,s,d}(\rn)$,
we now recall the following definition of $s$-order modified
Calder\'on--Zygmund operators (see, for instance,
\cite[Definition 2.11]{jtyyz3}). In what follows, for any given $s\in\mathbb{Z}_+$ and
$\dz\in(0,1]$, we always need to consider a function
$f\in L^1_{\loc}(\rn)$ satisfying that,
for any ball $B(x,r)\in\mathbb{B}(\rn)$ with $x\in\rn$ and $r\in(0,\fz)$,
\begin{align}\label{suita}
\int_{\rn\setminus B(x,r)}\frac{|f(y)-P_{B(x,r)}^{(s)}(f)(y)|}{|x-y|^{n+s+\dz}}\,dy<\fz.
\end{align}

\begin{definition}\label{def-B-CZO}
Let $s\in\zz_+$, $\dz\in(0,1]$, and $K$ be an $s$-order standard kernel.
For any $x,y\in\rn$ with $x\neq y$, let
\begin{align}\label{Kw-K}
\widetilde{K}(x,y):=K(y,x),
\end{align}
which is the formal adjoint kernel of $K$,
and $B_0:=B(x_0,r_0)\in\mathbb{B}(\rn)$ with center $x_0\in\rn$ and radius $r_0\in(0,\fz)$.
The \emph{$s$-order modified Calder\'on--Zygmund
operator $\widetilde{T}_{B_0}$ with kernel $\widetilde{K}$}
is defined by setting, for any $f\in L^1_{\mathrm{loc}}(\rn)$
satisfying \eqref{suita}
and for almost every $x\in \rn$,
\begin{align}\label{2.12x}
\widetilde{T}_{B_0}(f)(x):=
\lim_{\eta\to 0^+}\widetilde{T}_{B_0,\eta}(f)(x),
\end{align}
where, for any $\eta\in(0,\fz)$,
\begin{align*}
\widetilde{T}_{B_0,\eta}(f)(x)
:&=\int_{\rn\setminus B(x,\eta)}
\lf[\widetilde{K}(x,y)\phantom{\sum_{\{\gamma\in\zz_+^n:\ |\gamma|\leq s\}}\frac{\partial_{(1)}^{\gamma}\widetilde{K}}{\gamma!}}\r.\\
&\quad\lf.-\sum_{\{\gamma\in\zz_+^n:\ |\gamma|\leq s\}}
\frac{\partial_{(1)}^{\gamma}\widetilde{K}(x_0,y)}{\gamma!}
(x-x_0)^{\gamma}\mathbf{1}_{\rn\setminus B_0}(y)\r]f(y)\,dy.
\end{align*}
\end{definition}

Next, we show that $\widetilde{T}_{B_0}(f)$
in \eqref{2.12x} is well defined;
since its proof is quite similar to that of
\cite[Proposition 2.22]{jtyyz3}, we only
sketch some important steps.

\begin{proposition}\label{converge-B}
Let all the symbols be the same as in Definition \ref{def-B-CZO}.
Then, for any $f\in L^1_{\loc}(\rn)$ satisfying \eqref{suita},
$\widetilde{T}_{B_0}(f)$ in \eqref{2.12x}
is well defined almost everywhere on $\rn$.
\end{proposition}

\begin{proof}
Let $f\in L^1_{\loc}(\rn)$ satisfy \eqref{suita},
and $\widetilde{B}:=B(z,r)\in\mathbb{B}(\rn)$ with $z\in\rn$ and
$r\in(0,\fz)$.
To show the present proposition, we write,
for any given $\eta\in(0,r)$ [and hence $B(x,\eta)\subset
2\widetilde{B}$]
and for any $x\in \widetilde{B}$,
$$\widetilde{T}_{B_0,\eta}(f)(x)
=E_{\widetilde{B},\eta}(x)+E_{\widetilde{B}}(x)
+E_{\widetilde{B},B_0,\eta}^{(1)}(x)+E_{\widetilde{B},B_0,\eta}^{(2)}(x),$$
where
\begin{align*}
E_{\widetilde{B},\eta}(x):=\int_{2\widetilde{B}\setminus B(x,\eta)}
\widetilde{K}(x,y)\lf[f(y)-P_{2\widetilde{B}}^{(s)}(f)(y)\r]\,dy,
\end{align*}
\begin{align}\label{Eb}
E_{\widetilde{B}}(x):&=\int_{\rn\setminus2\widetilde{B}}
\lf[\widetilde{K}(x,y)-\sum_{\{\gamma\in\zz_+^n:\ |\gamma|\leq s\}}
\frac{\partial_{(1)}^{\gamma}\widetilde{K}(z,y)}{\gamma!}(x-z)^{\gamma}\r]\noz\\
&\quad\times\lf[f(y)-P_{2\widetilde{B}}^{(s)}(f)(y)\r]\,dy,
\end{align}
\begin{align*}
E_{\widetilde{B},B_0,\eta}^{(1)}(x)
:=\int_{\rn\setminus B(x,\eta)}L_{\widetilde{B},B_0}(x,y)\lf[f(y)
-P_{2\widetilde{B}}^{(s)}(f)(y)\r]\,dy
\end{align*}
with
\begin{align*}
L_{\widetilde{B},B_0}(x,y)
:&=\sum_{\{\gamma\in\zz_+^n:\ |\gamma|\leq s\}}
\frac{\partial_{(1)}^{\gamma}\widetilde{K}(z,y)}{\gamma!}(x-z)^{\gamma}
\mathbf1_{\rn\setminus 2\widetilde{B}}(y)\\
&\quad-\sum_{\{\gamma\in\zz_+^n:\ |\gamma|\leq s\}}
\frac{\partial_{(1)}^{\gamma}\widetilde{K}(x_0,y)}{\gamma!}(x-x_0)^{\gamma}
\mathbf1_{\rn\setminus B_0}(y),
\end{align*}
and
\begin{align*}
E_{\widetilde{B},B_0,\eta}^{(2)}(x)
:&=\int_{\rn\setminus B(x,\eta)}
\lf[\widetilde{K}(x,y)\phantom{\sum_{\{\gamma\in\zz_+^n:\ |\gamma|\leq s\}}
	\frac{\partial_{(1)}^{\gamma}\widetilde{K}(x_0,y)}{\gamma!}}\r.\\
&\quad\lf.-\sum_{\{\gamma\in\zz_+^n:\ |\gamma|\leq s\}}
\frac{\partial_{(1)}^{\gamma}\widetilde{K}(x_0,y)}{\gamma!}(x-x_0)^{\gamma}
\mathbf1_{\rn\setminus B_0}(y)\r] P_{2\widetilde{B}}^{(s)}(f)(y)\,dy.
\end{align*}
Thus, to prove that $\widetilde{T}_{B_0}(f)$ is well defined
almost everywhere on $\wz B$, it suffices to show that
$E_{\widetilde{B},\eta}$, $E_{\widetilde{B}}$,
$E_{\widetilde{B},B_0,\eta}^{(1)}$,
and $E_{\widetilde{B},B_0,\eta}^{(2)}$ are convergent
almost everywhere on $B$ as $\eta\to0^+$, respectively.

We first consider $E_{\widetilde{B},\eta}$.
Indeed, it is easy to show that
$[f-P_{2\widetilde{B}}^{(s)}(f)]\mathbf1_{2\widetilde{B}}\in L^1(\rn)$
and hence, for any $\eta\in(0,\fz)$, $E_{\widetilde{B},\eta}$ is well defined almost everywhere
on $\wz B$.
Moreover, let $T_{(\widetilde{K})}$ be the $s$-order
Calder\'on--Zygmund
singular integral operator
with kernel $\widetilde{K}$. Then,
by Remark \ref{rem-2.15}, we know that
\begin{align}\label{T1}
\lim_{\eta\to0^+}E_{\widetilde{B},\eta}
=T_{(\widetilde{K})}\lf(\lf[f-P_{2\widetilde{B}}^{(s)}(f)\r]
\mathbf1_{2\widetilde{B}}\r)
\end{align}
almost everywhere on $\widetilde{B}$.

Now, we estimate $E_{\widetilde{B}}$.
Observe that, since $E_{\widetilde{B}}$ is independent of $\eta$,
we only need to prove that $E_{\widetilde{B}}$ is bounded on
$\widetilde{B}$.
Indeed, applying both an argument similar to that used in the estimation of
\cite[(2.33)]{jtyyz3} and \eqref{suita},
we conclude that, for any $x\in\widetilde{B}$,
\begin{align}\label{lem-CZ-01}
\lf|E_{\widetilde{B}}(x)\r|	
&\ls r^{s+\dz}\int_{\rn\setminus 2\widetilde{B}}\frac{|f(y)
-P_{2\widetilde{B}}^{(s)}(f)(y)|}{|y-z|^{n+s+\dz}}\,dy<\fz.
\end{align}

Next, we consider both $E_{\widetilde{B},B_0,\eta}^{(1)}$ and
$E_{\widetilde{B},B_0,\eta}^{(2)}$.
Indeed, from \cite[(2.39)]{jtyyz3}, we deduce that
\begin{align}\label{T3-lim}
P_{1}(\widetilde{B},B_0;f)(\cdot)
:=&\lim_{\eta\to0^+}E_{\widetilde{B},B_0,\eta}^{(1)}(\cdot)
\in \mathcal{P}_s(B).
\end{align}
Moreover, similarly to the proof of \cite[(2.40)]{jtyyz3}, we conclude that
there exists a polynomial $P_2(\widetilde{B},B_0;f)\in
\mathcal{P}_s(\rn)$
such that, for almost every $x\in \widetilde{B}$,
\begin{align}\label{T4}
\lim_{\eta\to 0^+}E_{\widetilde{B},B_0,\eta}^{(2)}(x)
=P_2(\widetilde{B},B_0;f)(x).
\end{align}
From \eqref{T1}, \eqref{lem-CZ-01}, \eqref{T3-lim}, and \eqref{T4},
we deduce that, for any given ball $\widetilde{B}\in\mathbb{B}(\rn)$
and for almost every $x\in \widetilde{B}$,
\begin{align}\label{limit}
\widetilde{T}_{B_0}(f)(x)
&=\lim_{\eta\to0^+}\widetilde{T}_{B_0,\eta}(f)(x)\noz\\
&=T_{(\widetilde{K})}\lf(\lf[f-P_{2\widetilde{B}}^{(d)}(f)
\r]\mathbf1_{2\widetilde{B}}\r)(x)
+E_{\widetilde{B}}(x)\noz\\
&\quad+P_1(\widetilde{B},B_0;f)(x)
+P_2(\widetilde{B},B_0;f)(x).
\end{align}
Therefore, $\widetilde{T}_{B_0}(f)$
is well defined almost everywhere on $\rn$,
which then completes the proof of Proposition \ref{converge-B}.
\end{proof}

\begin{remark}\label{rem-CZ-B}
Let all the symbols be the same as in Proposition \ref{converge-B}.
By \cite[Remark 2.23]{jtyyz3}, we conclude that,
for any $f$ satisfying \eqref{suita}
and for any ball $B_1\in \mathbb{B}(\rn)$,
$$
\widetilde{T}_{B_0}(f)-\widetilde{T}_{B_1}(f)\in \mathcal{P}_s(\rn)
$$
after changing values on a set of measure zero.
Based on this, in what follows, we write $\widetilde{T}$
instead of $\widetilde{T}_{B_0}$ if there exists no confusion.
\end{remark}

Now, we show that $\wz T$ is well defined on both
$\mathcal{L}_{X,q,s}(\rn)$ and $\mathcal{L}_{X,q,s,d}(\rn)$.
To this end, we need the following mild assumption
about the boundedness of the Hardy--Littlewood maximal operator
on ball quasi-Banach function spaces (see, for instance, \cite[Assumption 2.6]{zhyy2022}).
Recall that the \emph{Hardy-Littlewood maximal operator}
$\mathcal{M}$ is defined by setting, for any $f\in
L_{\mathrm{loc}}^1(\rn)$ and $x\in\rn$,
\begin{align*}
\mathcal{M}(f)(x):=\sup_{B\ni x}\frac{1}{|B|}\int_{B}|f(y)|dy,
\end{align*}
where the supremum is taken over all the balls $B\in\mathbb{B}(\rn)$ containing $x$.

\begin{assumption}\label{assump1}
Let $X$ be a ball quasi-Banach function space. Assume that there
exists a $p_-\in(0,\infty)$ such that,
for any given $p\in(0,p_-)$ and $u\in(1,\infty)$, there exists a
positive constant $C$ such that,
for any $\{f_j\}_{j=1}^\infty\subset\mathscr M(\rn)$,
\begin{align*}
\lf\|\lf\{\sum_{j\in\nn}\lf[\mathcal{M}(f_j)\r]^u\r\}^{\frac{1}{u}}
\r\|_{X^{1/p}}
\le C\lf\|\lf(\sum_{j\in\nn}|f_j|^u\r)^{\frac{1}{u}}\r\|_{X^{1/p}}.
\end{align*}
\end{assumption}

\begin{remark}\label{main-remark}
Let both $X$ and $p_-$ satisfy Assumption \ref{assump1}, and
let $d\in(0,\fz)$.
From this and the fact that, for any $r\in(0,\min\{d,p_-\})$, any ball $B\in\mathbb{B}(\rn)$,
and any $\beta\in[1,\fz)$,
$\mathbf{1}_{\beta B}\leq (\beta+1)^{\frac{dn}{r}}
[\mathcal{M}(\mathbf{1}_B)]^{\frac{d}{r}},$
we easily deduce that, for any $r\in(0,\min\{d,p_-\})$, any
$\beta\in[1,\fz)$,
any sequence $\{B_j\}_{j\in\nn}\subset \mathbb{B}(\rn)$,
and any $\{\lambda_j\}_{j\in\nn}\subset [0,\fz)$,
\begin{align*}
\lf\|\lf(\sum_{j\in\nn}\lambda_j^d\mathbf{1}_{\beta
B_j}\r)^{\frac{1}{d}}\r\|_{X}
\leq(2\beta)^{\frac{n}{r}}\lf\|\lf(\sum_{j\in\nn}
\lambda_j^d\mathbf{1}_{B_j}\r)^{\frac{1}{d}}\r\|_{X}.
\end{align*}
Particularly, for any $r\in(0,p_-)$ and any ball $B\in
\mathbb{B}(\rn)$, we have
\begin{align}\label{key-c}
\lf\|\mathbf{1}_{\beta B}\r\|_{X}
\leq (2\beta)^{\frac{n}{r}}\|\mathbf{1}_{B}\|_{X}.
\end{align}
\end{remark}

We also need the following lemma which is a direct corollary of \cite[Lemma 3.12]{cjy-02};
we omit the details here.

\begin{lemma}\label{lem-suit}
Let both $X$ and $p_-$ satisfy Assumption \ref{assump1}.	
Let $\dz\in(0,1]$, $s\in(\frac{n}{p_-}-n-\dz,\fz)\cap\zz_+$, $q\in[1,\fz)$, and $d\in(0,\fz)$.
Then, for any $f\in \mathcal{L}_{X,q,s}(\rn)$ [resp., $f\in\mathcal{L}_{X,q,s,d}(\rn)$],
$f$ satisfies \eqref{suita}.
\end{lemma}

As a direct application of both Lemma \ref{lem-suit}
and Proposition \ref{converge-B}, we immediately obtain the following conclusion;
we omit the details here.

\begin{theorem}\label{thm-well-def-T}
Let both $X$ and $p_-$ satisfy Assumption
\ref{assump1}.	
Let $\dz\in(0,1]$, $q\in(1,\fz)$, $s\in(\frac{n}{p_-}-n-\dz,\fz)\cap\zz_+$, and $d\in(0,\fz)$.
Let $\wz T$ be the same as in Remark
\ref{rem-CZ-B}.
Then $\wz T$ is well defined on both $\mathcal{L}_{X,q,s}(\rn)$
and $\mathcal{L}_{X,q,s,d}(\rn)$.
\end{theorem}

Next, we establish the boundedness of $\wz T$, respectively, on
$\mathcal{L}_{X,q,s}(\rn)$
and $\mathcal{L}_{X,q,s,d}(\rn)$. To be precise, we show the following two theorems.

\begin{theorem}\label{thm-B-CZ1}
Let both $X$ and $p_-$
satisfy Assumption \ref{assump1}.	
Let $\dz\in(0,1]$, $q\in(1,\fz)$, $s\in(\frac{n}{p_-}-n-\dz,\fz)\cap\zz_+$,
$K$ be an $s$-order standard kernel with regularity $\dz$,
$T$ the same as in Definition \ref{def-s-k} with kernel $K$,
$\wz K$ the same as in \eqref{Kw-K}, and $\wz T$ the same
as in Remark \ref{rem-CZ-B} with
kernel $\wz K$.
Then $\widetilde{T}$ is bounded on
$\mathcal{L}_{X,q,s}(\rn)$, namely,
there exists a positive constant $C$ such that,
for any $f\in \mathcal{L}_{X,q,s}(\rn)$,
$$\lf\|\widetilde{T}(f)\r\|_{\mathcal{L}_{X,q,s}(\rn)}
\le C \|f\|_{\mathcal{L}_{X,q,d}(\rn)}$$
if and only if $T$ has the vanishing moments up to order $s$.
\end{theorem}

\begin{theorem}\label{thm-B-CZ}
Let both $X$ and $p_-$
satisfy Assumption \ref{assump1}.
Let $\dz\in(0,1]$, $q\in(1,\fz)$, $d\in(0,\infty)$,
$s\in(\frac{n}{\min\{p_-,d\}}-n-\dz,\fz)\cap\zz_+$, $K$ be an $s$-order standard kernel
with regularity $\dz$,
$T$ the same as in Definition \ref{def-s-k} with kernel $K$,
$\wz K$ the same as in \eqref{Kw-K}, and $\wz T$ the same
as in Remark \ref{rem-CZ-B} with
kernel $\wz K$.
Then $\widetilde{T}$ is bounded on
$\mathcal{L}_{X,q,s,d}(\rn)$, namely,
there exists a positive constant $C$ such that,
for any $f\in \mathcal{L}_{X,q,s,d}(\rn)$,
$$\lf\|\widetilde{T}(f)\r\|_{\mathcal{L}_{X,q,s,d}(\rn)}
\le C \|f\|_{\mathcal{L}_{X,q,s,d}(\rn)}$$
if and only if $T$ has the vanishing moments up to order $s$.
\end{theorem}

We only prove Theorem \ref{thm-B-CZ}; since
the proof of Theorem \ref{thm-B-CZ1} is similar, we omit the details here.
To this end, we first state two technique lemmas. The following conclusion
is just \cite[Proposition 2.16]{jtyyz3}.

\begin{lemma}\label{Assume}
Let $s\in\zz_+$, $K$ be an $s$-order standard kernel, and $T$ an
$s$-order Calder\'on--Zygmund singular
integral operator with kernel $K$.
Let $\widetilde{K}$ be the same as in \eqref{Kw-K},
$B_0:=B(x_0,r_0)\subset \rn$ a given ball
with $x_0\in\rn$ and $r_0\in(0,\fz)$,
and, for any $\nu\in\zz_+^n$ with $|\nu|\leq s$,
$\widetilde{T}_{B_0}(y^{\nu})$ the same as in \eqref{2.12x} with $f$
replaced by $y^{\nu}$.
Then $T$ has the vanishing moments up to order $s$ if and only if,
for any $\nu\in\zz_+^n$ with $|\nu|\leq s$,
$\widetilde{T}_{B_0}(y^{\nu})(\cdot)\in\mathcal{P}_s(\rn)$
after changing values on a set of measure zero.
\end{lemma}

Similarly to the proof of
\cite[Lemma 2.21]{jtyyz3}, we have the following lemma; we omit the details here.

\begin{lemma}\label{I-JN}
Let $q\in[1,\fz)$, $s\in\zz_+$, and $\lambda\in(s,\fz)$.
Then there exists a positive constant $C$ such that,
for any $f\in L^1_{\loc}(\rn)$ and
any ball $B(x,r)$ with $x\in\rn$ and $r\in(0,\fz)$,
\begin{align*}
&\int_{\rn\setminus B(x,r)}\frac{|f(y)-P_{B(x,r)}^{(s)}(f)(y)|}{|x-y|^{n+\lz}}\,dy\\
&\quad\leq C\sum_{k\in\nn}\lf(2^kr\r)^{-\lambda}
\lf[\fint_{2^{k}B(x,r)}\lf|f(y)
-P_{2^{k}B(x,r)}^{(s)}(f)(y)\r|^q\,dy\r]^{\frac{1}{q}}.
\end{align*}
\end{lemma}

Now, we show Theorem \ref{thm-B-CZ}.

\begin{proof}[Proof of Theorem \ref{thm-B-CZ}]
From Theorem \ref{thm-well-def-T}, we deduce that $\widetilde{T}$
is well defined on $\mathcal{L}_{X,q,s,d}(\rn)$.
We now show the necessity. Indeed, if $\widetilde{T}$
is bounded on $\mathcal{L}_{X,q,s,d}(\rn)$,
then, for any $\gamma\in\zz_+^n$ with $|\gamma|\leq s$,
$$
\lf\|\widetilde{T}(x^{\gamma})\r\|_{\mathcal{L}_{X,q,s,d}(\rn)}
\ls \lf\|x^{\gamma}\r\|_{\mathcal{L}_{X,q,s,d}(\rn)}=0
$$
and hence $\widetilde{T}(x^{\gamma})\in \mathcal{P}_s(\rn)$.
Using this and Lemma \ref{Assume}, we find that,
for any $\gamma\in\zz_+^n$ with $|\gamma|\leq s$,
$T^*(x^{\gamma})=0$,
which completes the proof of the necessity.

Next, we show the sufficiency.
To this end, let $T_{(\widetilde{K})}$ be the $s$-order
Calder\'on--Zygmund
singular integral operator with kernel $\widetilde{K}$, $m\in\nn$,
$\{B_j\}_{j=1}^m:=\{B(z_j,r_j)\}_{j=1}^m\subset\mathbb{B}(\rn)$
with both $\{z_j\}_{j=1}^m\subset\rn$ and $\{r_j\}_{j=1}^{m}\subset(0,\fz)$, and
$\{\lambda_j\}_{j=1}^m\subset[0,\infty)$ with
$\sum_{j=1}^m\lambda_j\neq0$.
It is easy to prove that,
for any $g\in L^1_{\mathrm{loc}}(\rn)$ and any ball
$B\in\mathbb{B}(\rn)$,
\begin{align*}
\lf[\fint_B\lf|g(x)-P_B^{(s)}(g)(x)
\r|^q\,dx\r]^{\frac{1}{q}}
\sim\inf_{P\in \mathcal{P}_s(B)}
\lf[\fint_B|g(x)-P(x)|^q\,dx\r]^{\frac{1}{q}};
\end{align*}
see, for instance, \cite[(2.12)]{jtyyz1}.
From this, \eqref{limit}, and $P_{B}^{(s)}(P)=P$ for any
$P\in \mathcal{P}_s(\rn)$ and any ball
$B\in\mathbb{B}(\rn)$, we infer that,
for any given $f\in \mathcal{L}_{X,q,s,d}(\rn)$,
\begin{align}\label{3.16}
&\sum_{j=1}^m\frac{{\lambda}_j|B_j|}{\|{\mathbf{1}}_{B_j}\|_{X}}
\lf[\fint_{B_j}\lf|f(x)-P^{(s)}_{B_j}(f)(x)\r|^q
\,dx\r]^\frac1q\noz\\
&\quad\ls\sum_{j=1}^m\frac{{\lambda}_j|B_j|}
{\|{\mathbf{1}}_{B_j}\|_{X}}
\lf[\fint_{B_j}\lf|T_{(\widetilde{K})}\lf(\lf[f-P_{2B_j}^{(s)}(f)\r]
\mathbf1_{2B_j}\r)(x)\r|^q\,dx\r]^{\frac1q}\noz\\
&\qquad+\sum_{j=1}^m\frac{{\lambda}_j|B_j|}
{\|{\mathbf{1}}_{B_j}\|_{X}}
\lf[\fint_{B_j}\lf|E_{B_j}(x)\r|^q\,dx\r]^{\frac1q}\noz\\
&\quad=:\mathrm{I}_1+\mathrm{I}_2,
\end{align}
where $E_{B_j}$ is the same as in \eqref{Eb} with $\wz B$ therein replaced by
$B_j$ here.

Now, we separately estimate $\mathrm{I}_1$ and $\mathrm{I}_2$.
Indeed, by Remarks \ref{rem-2.15} and \ref{main-remark}, we find that
\begin{align}\label{3.17}
\mathrm{I}_1\ls\sum_{j=1}^m\frac{{\lambda}_j|2B_j|}
{\|{\mathbf{1}}_{2B_j}\|_{X}}
\lf[\fint_{2B_j}\lf|f(x)-P_{2B_j}^{(s)}(f)(x)\r|^q\,dx\r]^{\frac1q}.
\end{align}
This is a desired estimate of $\mathrm{I}_1$.
Besides, for any $j\in\{1,\ldots,m\}$ and $k\in\nn$, let
$\lambda_{j,k}:=\frac{\lambda_j\|\mathbf{1}_{2^kB_j}\|_{X}}
{\|\mathbf{1}_{B_j}\|_{X}}$.
Then, from both \eqref{lem-CZ-01} and Lemma \ref{I-JN}
with $\lambda:=s+\dz$, we deduce that,
for any given $r\in(0,p_-)$,
\begin{align}\label{3.18}
\mathrm{I}_2
&\ls\sum_{j=1}^m\frac{{\lambda}_j|B_j|}{\|{\mathbf{1}}_{B_j}\|_{X}}
r_j^{s+\dz}\int_{\rn\setminus 2B_j}\frac{|f(y)
-P_{2B_j}^{(s)}(f)(y)|}{|y-z_j|^{n+d+\dz}}\,dy\noz\\
&\ls\sum_{j=1}^m\frac{{\lambda}_j|B_j|}{\|{\mathbf{1}}_{B_j}\|_{X}}
\sum_{k\in\nn}\frac1{2^{k(s+\dz)}}
\lf[\fint_{2^{k+1}B_j}\lf|f(y)-P_{2^{k+1}B_j}^{(s)}(f)(y)
\r|^q\,dy\r]^{\frac{1}{q}}\noz\\
&\sim\sum_{k\in\nn}2^{k(-n-s-\dz)}\sum_{j=1}^{m}
\frac{{\lambda}_{j,k}|2^{k+1}B_j|}{\|{\mathbf{1}}_{2^{k+1}B_j}\|_{X}}\noz\\
&\quad\times\lf[\fint_{2^{k+1}B_j}\lf|f(y)-P_{2^{k+1}B_j}^{(s)}(f)(y)
\r|^q\,dy\r]^{\frac{1}{q}}.
\end{align}
This is a desired estimate of $\mathrm{I}_2$.

Moreover, by both Remark \ref{main-remark} and the definition of
$\lambda_{i,k}$
with both $i\in\{1,\ldots,m\}$ and $k\in\nn$, we conclude that,
for any given $r\in(0,\min\{d,p_-\})$ and for any $k\in\nn$,
\begin{align*}
\lf\|\lf\{\sum_{i=1}^m\lf[\frac{\lambda_{i,k}}
{\|\mathbf{1}_{2^kB_i}\|_{X}}\r]^d
\mathbf{1}_{2^kB_i}\r\}^{\frac{1}{d}}\r\|_{X}
&\ls2^{\frac{kn}{r}}\lf\|\lf\{\sum_{i=1}^m
\lf[\frac{\lambda_{i,k}}{\|\mathbf{1}_{2^kB_i}\|_{X}}\r]^d
\mathbf{1}_{B_i}\r\}^{\frac{1}{d}}\r\|_{X}\\
&\sim2^{\frac{kn}{r}}\lf\|\lf\{\sum_{i=1}^m
\lf[\frac{\lambda_{i}}{\|\mathbf{1}_{B_i}\|_{X}}\r]^d
\mathbf{1}_{B_i}\r\}^{\frac{1}{d}}\r\|_{X}.
\end{align*}
From this, \eqref{3.16}, \eqref{3.17}, \eqref{3.18},
and $s\in(\frac{n}{\min\{p_-,d\}}-n-\dz,\fz)$
which further implies that $\min\{d,p_-\}\in(\frac{n}{n+s+\dz},\fz)$,
we deduce that,
for any given $r\in(\frac{n}{n+s+\dz},\min\{d,p_-\})$,
\begin{align*}
&\lf\|\lf\{\sum_{i=1}^m\lf[\frac{{\lambda}_i}{\|{\mathbf{1}}_{B_i}\|_X}\r]^d
{\mathbf{1}}_{B_i}\r\}^{\frac1d}\r\|_{X}^{-1}
\sum_{j=1}^m\frac{{\lambda}_j|B_j|}{\|{\mathbf{1}}_{B_j}\|_{X}}
\lf[\frac1{|B_j|}\int_{B_j}\lf|f(x)-P^{(s)}_{B_j}(f)(x)\r|^q
\,dx\r]^\frac1q\\
&\quad\ls\lf\|\lf\{\sum_{i=1}^m\lf[\frac{\lambda_{i}}
{\|\mathbf{1}_{2B_i}\|_{X}}\r]^d
\mathbf{1}_{2B_i}\r\}^{\frac{1}{d}}\r\|_{X}^{-1}\\
&\qquad\times\sum_{j=1}^m\frac{{\lambda}_j|2B_j|}{\|{\mathbf{1}}_{2B_j}\|_{X}}
\lf[\fint_{2B_j}\lf|f(x)-P_{2B_j}^{(s)}(f)(x)\r|^q\,dx\r]^{\frac1q}\\
&\qquad+\sum_{k\in\nn}2^{k(-n-s-\dz+\frac{n}{r})}
\lf\|\lf\{\sum_{j\in\nn}\lf[\frac{\lambda_{i}}
{\|\mathbf{1}_{2^kB_i}\|_{X}}\r]^d
\mathbf{1}_{2^kB_i}\r\}^{\frac{1}{d}}\r\|_{X}^{-1}\\
&\qquad\times\lf\{\sum_{j=1}^{m}
\frac{{\lambda}_j|2^{k+1}B_j|}{\|{\mathbf{1}}_{2^{k+1}B_j}\|_{X}}
\lf[\fint_{2^{k+1}B_j}\lf|f(y)-P_{2^{k+1}B_j}^{(s)}(f)(y)
\r|^q\,dy\r]^{\frac{1}{q}}\r\}\\
&\quad\ls\|f\|_{\mathcal{L}_{X,q,s,d}(\rn)}.
\end{align*}
This, together with the definition of $\|\cdot\|_{\mathcal{L}_{X,q,s,d}(\rn)}$,
then further implies the desired conclusion, which completes
the proof of the sufficiency and
hence Theorem \ref{thm-B-CZ}.
\end{proof}

\subsection{Relations between $\wz T$ and $T$}\label{sec-CZO-B-02}

In this subsection, we show that $\wz T$
is just the adjoint operator of $T$.
To this end, we first recall the
concept of the Hardy space
$H_X(\rn)$ associated with $X$, which was originally studied in \cite{SHYY}.

In what follows, denote by $\mathcal{S}(\rn)$ the space of all
Schwartz functions equipped with the topology
determined by a well-known countable family of norms, and by
$\mathcal{S}'(\rn)$ its topological dual space equipped with the
weak-$*$ topology. For any $N\in\mathbb{N}$ and
$\phi\in\mathcal{S}(\rn)$, let
\begin{equation*}
p_N(\phi):=\sum_{\alpha\in\mathbb{Z}
^{n}_{+},|\alpha|\leq N}
\sup_{x\in\rn}(1+|x|)^{N+n}|
\partial^{\alpha}\phi(x)|
\end{equation*}
and
\begin{equation*}
\mathcal{F}_N(\rn):=\left\{\phi\in\mathcal{S}
(\rn):\ p_N(\phi)\in[0,1]\right\},
\end{equation*}
where, for any $\alpha:=(\alpha_1,\ldots,\alpha_n)\in\mathbb{Z}^n_+$,
$|\alpha|:=\alpha_1+\cdots+\alpha_n$ and
$\partial^\alpha:=(\frac{\partial}{\partial
x_1})^{\alpha_1}\cdots(\frac{\partial}{\partial
x_n})^{\alpha_n}$. Moreover, for any $r\in\rr$, we denote by $\lfloor r\rfloor$ (resp., $\lceil r\rceil$) the
\emph{maximal} (resp., \emph{minimal})
\emph{integer not greater} (resp., \emph{less}) \emph{than} $r$.

\begin{definition}\label{2d1}
Let $X$ be a ball quasi-Banach
function space and $N\in\nn$ be sufficiently
large. Then the \emph{Hardy
space} $H_X(\rn)$ is defined to be
the set of all the $f\in\mathcal{S}'(\rn)$
such that
$$
\left\|f\right\|_{H_X(\rn)}
:=\left\|\mathcal{M}_N(f)\right\|_{X}<\infty,
$$
where the \emph{non-tangential
grand maximal function}
$\mathcal{M}_{N}(f)$
of $f\in\mathcal{S}'(\rn)$
is defined by setting, for any $x\in\rn$,
\begin{align*}
\mathcal{M}_{N}(f)(x):=
\sup\left\{|f*\phi_{t}(y)|:\
\phi\in\mathcal{F}_{N}(\rn),\
t\in(0,\infty),\ |x-y|<t\right\}.
\end{align*}
\end{definition}

\begin{remark}
Let all the symbols be the same as in Definition \ref{2d1}.
Assume that there exists an $r\in(0,\fz)$ such that
the Hardy--Littlewood maximal operator $\mathcal{M}$ is bounded on $X^{\frac{1}{r}}$.
If $N\in[\lfloor \frac{n}{r}+1\rfloor,\fz)\cap\mathbb{N}$, then,
by \cite[Theorem 3.1]{SHYY}, we find that the Hardy space
$H_X(\rn)$ is independent of
the choice of $N$.
\end{remark}

To obtain the desired result, we still need to recall
some basic concepts.
In what follows, for any $\theta\in(0,\infty)$, the
\emph{powered Hardy--Littlewood
maximal operator} $\mathcal{M}^{(\theta)}$ is defined by setting,
for any measurable function $f$ and $x\in\rn$,
\begin{align*}
\mathcal{M}^{(\theta)}(f)(x):=\lf\{\mathcal{M}\lf(|f|^\theta\r)(x)\r\}^{\frac{1}{\theta}}.
\end{align*}
Moreover, the associate space $X'$ of any given ball
Banach function space $X$ is defined as follows
(see \cite[Chapter 1, Section 2]{BS88} or \cite[p.\,9]{SHYY}).

\begin{definition}\label{de-X'}
For any given ball quasi-Banach function space $X$, its \emph{associate space}
(also called the
\emph{K\"othe dual space}) $X'$ is defined by setting
\begin{equation*}
X':=\lf\{f\in\mathscr M(\rn):\ \|f\|_{X'}<\infty\r\},
\end{equation*}
where, for any $f\in X'$,
$$\|f\|_{X'}:=\sup\lf\{\lf\|fg\r\|_{L^1(\rn)}:\ g\in X,\ \|g\|_X=1\r\}$$
and $\|\cdot\|_{X'}$ is called the \emph{associate norm} of $\|\cdot\|_X$.
\end{definition}

Next, we recall the following mild assumption
about the boundedness of the powered Hardy--Littlewood maximal operator
on ball quasi-Banach function spaces (see, for instance, \cite[(2.9)]{SHYY}), which is used later in this article.

\begin{assumption}\label{assump2}
Let $X$ be a ball quasi-Banach function space.
Assume that there exists an $r_0\in(0,\infty)$ and a
$p_0\in(r_0,\infty)$
such that $X^{1/r_0}$ is a ball Banach function space and there exists a positive constant $C$ such that,
for any $f\in(X^{1/r_0})'$,
\begin{align*}
\lf\|\mathcal{M}^{((p_0/r_0)')}(f)\r\|_{(X^{1/r_0})'}\le
C\lf\|f\r\|_{(X^{1/r_0})'}.
\end{align*}
\end{assumption}

Now, we present the following definitions of both the
$(X,q,s)$-atom and the finite atomic Hardy space
$H_{\mathrm{fin}}^{X,q,s,d}(\rn)$ which are just, respectively,
\cite[Definition 3.5]{SHYY}
and \cite[Definition 1.9]{yyy20}.

\begin{definition}\label{atom}
Let $X$ be a ball quasi-Banach function space, $q\in(1,\infty]$, and
$s\in\zz_+$.
Then a measurable function $a$ on $\rn$ is called an
$(X,q,s)$-\emph{atom}
if there exists a ball $B\in\mathbb{B}(\rn)$ such that
\begin{enumerate}
\item[(i)] $\supp\,(a):=\{x\in\rn:\ a(x)\neq0\}\subset B$;
\item[(ii)]
$\|a\|_{L^q(\rn)}\le\frac{|B|^{\frac{1}{q}}}{\|\mathbf{1}_B\|_X}$;
\item[(iii)] $\int_{\rn} a(x)x^\gamma\,dx=0$ for any
$\gamma:=(\gamma_1,\ldots,\gamma_n)\in\zz_+^n$ with
$|\gamma|\le s$.
\end{enumerate}
\end{definition}

\begin{definition}\label{finatom}
Let both $X$ and $p_-$ satisfy Assumption \ref{assump1}. Assume that
$r_0\in(0,\min\{1,p_-\})$ and $p_0\in(r_0,\infty)$
satisfy Assumption \ref{assump2}.
Let $s\in[\lfloor n(\frac{1}{\min\{1,p_-\}}-1)\rfloor,\infty)\cap\zz_+$,
$d\in(0,r_0]$, and $q\in(\max\{1,p_0\},\infty]$.
The \emph{finite atomic Hardy space}
$H_{\mathrm{fin}}^{X,q,s,d}({{\rr}^n})$,
associated with $X$, is defined to be the set of all finite
linear combinations of $(X,q,s)$-atoms. The quasi-norm
$\|\cdot\|_{H_{\mathrm{fin}}^{X,q,s,d}({{\rr}^n})}$ in
$H_{\mathrm{fin}}^{X,q,s,d}({{\rr}^n})$
is defined by setting, for any $f\in
H_{\mathrm{fin}}^{X,q,s,d}({{\rr}^n})$,
\begin{align*}
\|f\|_{H_{\mathrm{fin}}^{X,q,s,d}({{\rr}^n})}&:=\inf\left\{\left\|
\left[\sum_{j=1}^{N}
\left(\frac{{\lambda}_j}{\|{\mathbf{1}}_{B_j}\|_X}\right)^d
{\mathbf{1}}_{B_j}\right]^{\frac1d}\right\|_{X}
\right\},
\end{align*}
where the infimum is taken over all finite linear combinations of $(X,q,s)$-atoms of
$f$, namely, $N\in\nn$, $f=\sum_{j=1}^{N}\lambda_ja_j$,
$\{\lambda_j\}_{j=1}^{N}\subset[0,\infty)$, and $\{a_j\}_{j=1}^{N}$
being
$(X,q,s)$-atoms supported, respectively,
in the balls $\{B_j\}_{j=1}^{N}\subset\mathbb{B}(\rn)$.
\end{definition}

Recall that $X$ is said to have an \emph{absolutely continuous
quasi-norm} if,
for any $f\in X$ and any measurable subsets $\{E_j\}_{j\in\nn}\subset \rn$
with $E_{j+1}\subset E_j$ for any $j\in\nn$ and $\bigcap_{j\in\nn}
E_j=\emptyset$,
$\|f\mathbf{1}_{E_j}\|_{X}\downarrow 0$ as $j\to\fz$.
The following lemma is just \cite[Theorem 3.14]{zhyy2022}.

\begin{lemma}\label{2t1}
Let $X$, $d$, $q$, and $s$ be the same as in Definition \ref{finatom},
and let $X$ have an absolutely continuous quasi-norm.
Then the dual space of $H_X({{\rr}^n})$, denoted by
$(H_X({{\rr}^n}))^*$,
is $\mathcal{L}_{X,q',s,d}({{\rr}^n})$ with
$\frac{1}{q}+\frac{1}{q'}=1$ in the following sense:
\begin{enumerate}
\item[{\rm (i)}] Let $g\in\mathcal{L}_{X,q',d,s}({{\rr}^n})$.
Then the linear functional
\begin{align}\label{2te1}
L_g:\ f\to \lf\langle L_g,f\r\rangle:=\int_{{{\rr}^n}}f(x)g(x)\,dx,
\end{align}
initially defined for any $f\in
H_{\mathrm{fin}}^{X,q,s,d}({{\rr}^n})$,
has a bounded extension to $H_X({{\rr}^n})$.

\item[{\rm (ii)}] Conversely, any continuous linear
functional on $H_X(\rn)$ arises the same as  in \eqref{2te1}
with a unique $g\in\mathcal{L}_{X,q',s,d}({{\rr}^n})$.
\end{enumerate}
Moreover,
$\|g\|_{\mathcal{L}_{X,q',s,d}({{\rr}^n})}\sim\|L_g\|_{(H_X({{\rr}^n}))^*}$
with the positive equivalence constants independent of both $g$ and $L_g$.
\end{lemma}

The following proposition is just a slight modification of
\cite[Theorem 3.11]{wyy}, which shows that
$T$ has a unique extension on $H_X(\rn)$; we omit the details here.

\begin{proposition}\label{lem-T-H}
Let both $X$ and $p_-$ satisfy Assumptions \ref{assump1} and \ref{assump2}.
Let $\dz\in(0,1]$, $s\in(\frac{n}{p_-}-n-\dz,\fz)$,
and $T$ be the same as in Definition \ref{Def-T-s-v}.
Further assume that $X$ has an absolutely continuous quasi-norm.
Then $T$ has a unique extension on
$H_X(\rn)$, still denoted by $T$, namely, there exists a
positive constant $C$ such that, for any $f\in H_X(\rn)$,
$$
\|T(f)\|_{H_X(\rn)}\leq C\|f\|_{H_X(\rn)}.
$$	
\end{proposition}

Next, we present the main result of this subsection.

\begin{theorem}\label{dual-T-wT}
Let $X$, $p_-$, $q$, and $d$ be the same as in Definition \ref{finatom}.	
Let $\dz\in(0,1]$,
$$s\in\lf(\frac{n}{d}-n-\dz,\fz\r)\cap\lf[\lf\lfloor n
\lf(\frac{1}{\min\{1,p_-\}}-1\r)\r\rfloor,\infty\r)\cap\zz_+,$$	
$T$ be the same as in
Proposition \ref{lem-T-H},
and $\wz T$ the same as in Remark \ref{rem-CZ-B}.
Further assume that $X$ has an absolutely continuous quasi-norm.
Then $\wz T$ is the adjoint operator of $T$ in the following sense:
for any $g\in \mathcal{L}_{X,q',s,d}(\rn)$ and $f\in H_X(\rn)$,
\begin{align*}
\lf\langle L_g,T(f)\r\rangle=\lf\langle L_{\wz T(g)},f\r\rangle,
\end{align*}	
where $L_g$ is the same as in \eqref{2te1}, $L_{\wz T(g)}$ the same
as in \eqref{2te1} with g replaced by $\wz T(g)$, and
$\frac{1}{q}+\frac{1}{q'}=1$.
\end{theorem}

To prove this theorem, we first establish three technique lemmas.
We begin with showing that, for any $(X,q,s)$-atom $a$,
$T(a)$ is a harmless constant multiple of an $(X,q,s,\tau)$-molecule
if and only if $T$ has the vanishing moments up to order $s$.
To this end, we first recall the following definition of $(X,q,s,\tau)$-molecules,
which is just \cite[Definition 3.8]{SHYY}.

\begin{definition}\label{def-mol}
Let $X$ be a ball quasi-Banach function space, $q\in(1,\infty]$,
$s\in\zz_+$, and $\tau\in(0,\fz)$. Then a measurable function $m$ on
$\rn$ is called an
\emph{$(X,q,s,\tau)$-molecule} centered at a ball
$B\in \mathbb{B}(\rn)$ if
\begin{enumerate}
\item[\rm(i)] for any $j\in\zz_+$,
\begin{align*}
\lf\|m\mathbf{1}_{L_j}\r\|_{L^q(\rn)}
\leq 2^{-j\tau}\frac{|B|^{\frac{1}{q}}}{\|\mathbf{1}_{B}\|_{X}},
\end{align*}
where $L_0:=B$ and, for any $j\in\nn$, $L_j:=2^{j}B\setminus
2^{j-1}B$;
\item[\rm(ii)] $\int_{\rn} m(x)x^\gamma\,dx=0$ for any
$\gamma\in\zz_+^n$ with $|\gamma|\le s$.
\end{enumerate}
\end{definition}

\begin{lemma}\label{lem-Ta-mole}
Let both $X$ and $p_-$ satisfy Assumption \ref{assump1}.
Let $s\in\zz_+$, $\dz\in(0,1]$, $T$ be the same as in Definition \ref{def-s-k},
$q\in(1,\fz)$, $\frac{1}{q}+\frac{1}{q'}=1$,
and $\tau\in(0,\frac{n}{q'}+s+\delta]$.
Then the following two statements are equivalent:
\begin{enumerate}
\item[\rm (i)]
there exists a positive constant $C$ such that,
for any $(X,q,s)$-atom $a$ supported in the ball $B\in \mathbb{B}(\rn)$,
$\frac{T(a)}{C}$ is an $(X,q,s,\tau)$-molecule centered at $2B$.
\item[\rm (ii)]
$T$ has the vanishing moments up to order $s$.
\end{enumerate}
\end{lemma}

\begin{proof}
We first show (i) $\Longrightarrow$ (ii).
To this end, let $A\in L^q(\rn)$ be supported in a
ball $B\in \mathbb{B}(\rn)$ and satisfy that,
for any $\gamma\in\zz_+^n$ with $|\gamma|\leq s$,
$\int_{\rn} A(x)x^\gamma\,dx=0$.
Without loss of generality, we may assume that $A$ satisfies that
$\|A\|_{L^q(\rn)}>0$.
Then it is easy to show that
$\widetilde{A}:=\frac{|B|^{\frac{1}{q}}A}{\|\mathbf{1}_B\|_{X}
\|A\|_{L^q(\rn)}}$
is an $(X,q,s)$-atom supported in $B$.
By this, (i) of the present lemma, and \cite[(2.17)]
{jtyyz3}, we conclude that,
for any $\gamma\in\zz_+^n$ with $|\gamma|\leq s$,
$$
\int_{B(x_0,r_0)}\widetilde{A}(x)\widetilde{T}(y^{\gamma})(x)\,dx
=\int_{\rn}T\lf(\widetilde{A}\r)(x)x^{\gamma}\,dx
=0.
$$
Applying this and an argument similar to that used in the proof of
\cite[Proposition 2.16]{jtyyz3},
we find that, for any $\gamma\in\zz_+^n$ with $|\gamma|\leq s$,
$\widetilde{T}(y^{\gamma})\in\mathcal{P}_s(\rn)$
after changing values on a set of measure zero,
which, combined with Proposition \ref{Assume}, further implies that
$T$ has the vanishing moments up to order $s$.
This finishes the proof that (i) $\Longrightarrow$ (ii).	

Next, we show (ii) $\Longrightarrow$ (i).
Let $a$ be an $(X,q,s)$-atom supported in the ball $B\in \mathbb{B}(\rn)$.
We prove that there exists a positive constant $C$,
independent of $a$, such that
$\frac{T(a)}{C}$ is an $(X,q,s,\tau)$-molecule centered at $2B$.
Indeed, using Remark \ref{rem-2.15}, the size condition of $a$,
and \eqref{key-c}, we find that
\begin{align}\label{Ta-i'}
\|T (a)\mathbf{1}_{2B}\|_{L^q(\rn)}
\lesssim \|a\|_{L^q(B)}
\lesssim \frac{|B|^{\frac{1}{q}}}{\|\mathbf{1}_B\|_{X}}
\sim\frac{|2B|^{\frac{1}{q}}}{\|\mathbf{1}_{2B}\|_{X}}
\end{align}
with the implicit positive constants independent of $a$.
Moreover, similarly to the estimation of $T(a)$ in
\cite[p.\,48]{jtyyz3},
we find that, for any $x\notin 2B$,
\begin{align*}
|T(a)(x)|\ls
\frac{r^{\frac{n}{q'}+s+\delta}}{|z-x|^{n+s+\delta}}\|a\|_{L^{q}(B)},
\end{align*}
which, together with both the size condition of $a$
and $\tau\in(0,\frac{n}{q'}+s+\delta]$,
further implies that, for any $j\in\nn$ and $x\in 2^{j+1}B
\setminus 2^{j}B$,
\begin{align*}
\lf|T(a)(x)\r|
&\ls2^{-j(\frac{n}{q'}+s+\delta)}
\lf(2^jr\r)^{-\frac{n}{q}}\frac{|B|^{\frac{1}{q}}}{\|\mathbf{1}_B\|_{X}}\ls
2^{-j\tau}\lf|2^{j+1}B\setminus 2^{j}B\r|^{-\frac{1}{q}}
\frac{|B|^{\frac{1}{q}}}{\|\mathbf{1}_B\|_{X}}.
\end{align*}
From this and \eqref{key-c}, we infer that, for any $j\in\nn$,
\begin{align*}
\lf\|T(a)\mathbf{1}_{2^{j+1}B\setminus 2^{j}B}\r\|_{L^{q}(\rn)}
\lesssim 2^{-j\tau}
\frac{|B|^{\frac{1}{q}}}{\|\mathbf{1}_B\|_{X}}
\sim 2^{-j\tau}
\frac{|2B|^{\frac{1}{q}}}{\|\mathbf{1}_{2B}\|_{X}}.
\end{align*}
This, combined with both \eqref{Ta-i'} and (ii) of the present lemma,
then finishes the proof that (ii) $\Longrightarrow$ (i) and hence the proof of
Lemma \ref{lem-Ta-mole}.
\end{proof}

The following lemma is just \cite[Lemma 3.30]{cjy-02}.

\begin{lemma}\label{thm-dual-T}
Let $X$, $p_-$, $d$, $q$, and $s$ be the same as in Definition \ref{finatom}.
Let
\begin{align*}
\tau\in \lf(n\lf[\frac{1}{p_-}-\frac{1}{q}\r],\fz\r)
\cap\lf(\frac{n}{q'}+s,\fz\r),
\end{align*}
where $\frac{1}{q}+\frac{1}{q'}=1$.
Further assume that $X$ has an absolutely continuous quasi-norm.
Then, for any $(X,q,s,\tau)$-molecule $m$ and any $g\in
\mathcal{L}_{X,q',s,d}(\rn)$,
\begin{align*}
\lf\langle L_g,m\r\rangle=\int_{\rn}m(x)g(x)\,dx,
\end{align*}
where $L_g$ is the same as in \eqref{2te1}.
\end{lemma}

The proof of the following lemma is a slight modification of
\cite[Lemma 3.19]{jtyyz3} and we only sketch some important steps.

\begin{lemma}\label{dual-T-a}
Let $X$, $p_-$, $d$, and $q$ be the same as in Definition \ref{finatom}.
Let $\dz\in(0,1]$,
$$s\in\lf(\frac{n}{d}-n-\dz,\fz\r)
\cap\lf[\lf\lfloor n\lf(\frac{1}{\min\{1,p_-\}}-1\r)\r\rfloor,\infty\r)\cap\zz_+,$$
$T$ be the same as in Definition \ref{Def-T-s-v},
and $\wz T$ the same as in Remark \ref{rem-CZ-B}.
Further assume that $X$ has an absolutely continuous quasi-norm.	
Then, for any $g\in \mathcal{L}_{X,q',s,d}(\rn)$
and any $(X,q,s)$-atom $a$,
\begin{align*}
\lf\langle L_g,T(a)\r\rangle=\lf\langle L_{\wz T(g)},a\r\rangle,
\end{align*}
where $L_g$ is the same as in \eqref{2te1}, $L_{\wz T(g)}$ the same
as in \eqref{2te1} with g replaced by $\wz T(g)$, and
$\frac{1}{q}+\frac{1}{q'}=1$.
\end{lemma}

\begin{proof}
By the ranges of both $d$ and $s$, we have $d\in(\frac{n}{n+s+\dz},\fz)$
and $p_-\in(\frac{n}{n+s+\dz},\fz)$. Thus, we can choose a
$\tau\in(n[\frac{1}{\min\{1,p_-\}}-\frac{1}{q}],\frac{n}{q'}+s+\delta]
\cap(\frac{n}{q'}+s,\frac{n}{q'}+s+\delta]$.
Then, using Lemma \ref{lem-Ta-mole}, we conclude that $T(a)$ is
a harmless constant multiple of
an $(X,q,s,\tau)$-molecule.
By this and Lemma \ref{thm-dual-T}, we find that, for any $g\in
\mathcal{L}_{X,q',s,d}(\rn)$
and any $(X,q,s)$-atom $a$,
\begin{align}\label{dual-ta-01}
\lf\langle L_g, T(a)\r\rangle=\int_{\rn}g(x)T(a)(x)\,dx.
\end{align}
Besides, using both $\min\{d,p_-\}\in(\frac{n}{n+s+\dz},\fz)$
and Theorem \ref{thm-B-CZ1},
we obtain $\wz T(g)\in \mathcal{L}_{X,q',s,d}(\rn)$,
which, combined with Lemma \ref{2t1},
further implies that, for any $(X,q,s)$-atom $a$,
\begin{align}\label{dual-ta-03}
\lf\langle L_{\wz T(g)},a\r\rangle=\int_{\rn}\wz T(g)(x)a(x)\,dx.
\end{align}
Moreover, by Lemma \ref{lem-suit} with $\lambda:=n+s+\dz$,
we conclude that, for any $g\in \mathcal{L}_{X,q',s,d}(\rn)$
and any  ball $B:=B(x_0,r_0)\in\mathbb{B}(\rn)$
with $x_0\in\rn$ and $r_0\in(0,\fz)$,
\begin{align}\label{dual-ta-02}
\int_{\rn\setminus B}\frac{|g(y)
-P_{B}^{(s)}(g)(y)|}{|x_0-y|^{n+s+\dz}}\,dy<\fz.
\end{align}
Using \eqref{dual-ta-02} and repeating the proof of
\cite[Lemma 3.19]{jtyyz3}, we conclude that
\begin{align*}
\int_{\rn}g(x)T(a)(x)\,dx=\int_{\rn}\wz T(g)(x)a(x)\,dx,
\end{align*}
which, together with both \eqref{dual-ta-01} and \eqref{dual-ta-03},
further implies that
\begin{align*}
\lf\langle L_g, T(a)\r\rangle=\int_{\rn}T(a)g(x)\,dx
=\int_{\rn}\wz T(g)(x)a(x)\,dx=\lf\langle L_{\wz T(g)},a\r\rangle.
\end{align*}
This finishes the proof of Lemma \ref{dual-T-a}.
\end{proof}

Now, we show Theorem \ref{dual-T-wT}.

\begin{proof}[Proof of Theorem \ref{dual-T-wT}]
To show the present theorem, let $f\in H_X(\rn)$.
Since $X$ has an absolutely continuous quasi-norm, using this and
\cite[Remark 3.12]{SHYY}, we find that
$H_{\mathrm{fin}}^{X,\fz,s,d}(\rn)$ is dense in $H_X(\rn)$.
Thus, there exists a sequence
$\{f_k\}_{k\in\nn}\subset H_{\mathrm{fin}}^{X,\fz,s,d}(\rn)$
such that $f_k\to f$ in $H_X(\rn)$.
Besides, by the definition of
$H_{\mathrm{fin}}^{X,\fz,s,d}(\rn)$, we conclude that,
for any $k\in\nn$,
there exists an $m_k\in\nn$, a sequence $\{a_j^{(k)}\}_{j=1}^{m_k}$ of $(X,\fz,s)$-atoms,
and $\{\lambda_j^{(k)}\}_{j=1}^{m_k}\subset[0,\fz)$ such that
\begin{align}\label{Tg-01}
f_k=\sum_{j=1}^{m_k}\lambda_j^{(k)}a_j^{(k)}.
\end{align}
Moreover, combining this and Proposition \ref{lem-T-H}, we have
$T(f_k)\to T(f)$ in $H_X(\rn)$ as $k\to\infty$.
From this, Lemma \ref{2t1}, \eqref{Tg-01}, and Lemmas
\ref{dual-T-a} and \ref{thm-B-CZ},
we deduce that, for any $g\in \mathcal{L}_{X,q',s,d}(\rn)$,
\begin{align*}
\lf\langle L_g,T(f)\r\rangle
&=\lim_{k\to\fz}\lf\langle L_g,T(f_k)\r\rangle
=\lim_{k\to\fz}\sum_{j=1}^{m_k}\lambda_{j}^{(k)}
\lf\langle L_g,T\lf(a_{j}^{(k)}\r)\r\rangle\\
&=\lim_{k\to\fz}\sum_{j=1}^{m_k}\lambda_{j}^{(k)}
\lf\langle L_{\wz T(g)},a_{j}^{(k)}\r\rangle
=\lim_{k\to\fz}\lf\langle L_{\wz T(g)},f_k\r\rangle
=\lf\langle L_{\wz T(g)},f\r\rangle.
\end{align*}
This finishes the proof of Theorem \ref{dual-T-wT}.
\end{proof}

\section{Applications\label{Appli}}

In this section, we apply all the above main results to
eight concrete examples of ball quasi-Banach
function spaces, namely, weighted Lebesgue spaces (Subsection \ref{weighted}),
variable Lebesgue spaces (Subsection \ref{variable}),
Orlicz spaces (Subsection \ref{Orlicz}),
Orlicz-slice spaces (Subsection \ref{Orlicz-Slice}), Morrey spaces
(Subsection \ref{Morrey}), mixed-norm Lebesgue spaces (Subsection \ref{Mixed-Norm}),
local generalized Herz spaces (Subsection \ref{H-H}),
and mixed Herz spaces (Subsection \ref{M-H}).
Moreover, to the best of our knowledge, all these results are new.
These applications explicitly indicate the generality and the practicability
of the main results of this article and more applications
to new function spaces are obviously possible.

\subsection{Weighted Lebesgue spaces\label{weighted}}

In this subsection, we apply
Theorems \ref{thm-B-CZ1}, \ref{thm-B-CZ}, and \ref{dual-T-wT}
to weighted Lebesgue spaces.
We first present the definitions of
both Muckenhoupt weights and weighted Lebesgue spaces
(see, for instance, \cite[Definitions 7.1.2 and 7.1.3]{G1}).

\begin{definition}
Let $p\in[1,\infty)$ and $w$ be a nonnegative locally integrable
function on $\rn$. Then
$w$ is called an \emph{$A_{p}(\rn)$ weight}, denoted by $w\in A_p(\rn)$, if, when $p\in(1,\infty)$,
\begin{align*}
[w]_{A_{p}\left(\mathbb{R}^{n}\right)}
:=\sup _{B \in \mathbb{B}(\rn)}\lf[\frac{1}{|B|}
\int_{B}w(x)\,dx\r]\left\{\frac{1}{|B|}
\int_{B}[w(x)]^{-\frac{1}{p-1}} \,dx\right\}^{p-1}<\infty,
\end{align*}
and
\begin{align*}
[w]_{A_{1}\left(\mathbb{R}^{n}\right)}
:=\sup _{B \in \mathbb{B}(\rn)}\lf[\frac{1}{|B|}
\int_{B}w(x)\,dx\r]\left\{\mathop{\mathrm{ess\,sup}}_{x\in\rn}
[w(x)]^{-1}\right\}<\infty.
\end{align*}
Moreover, the \emph{class $A_\infty(\rn)$} is defined by setting
$$
A_\infty(\rn):=\bigcup_{p\in[1,\infty)}A_p(\rn).
$$
\end{definition}

\begin{definition}\label{def-wei}
Let $p\in(0,\infty)$ and $w\in A_\infty(\rn)$.
The \emph{weighted Lebesgue space} $L_{w}^{p}(\mathbb{R}^{n})$ is
defined to be the set of all the measurable
functions $f$ on $\rn$ such that
\begin{align*}
\lf\|f\r\|_{L^p_w(\rn)}=\lf[\int_{\rn}|f(x)|^pw(x)\,dx\r]
^{\frac{1}{p}}<\infty.
\end{align*}
\end{definition}

\begin{remark}
It is worth pointing out that, for any $p\in(0,\fz)$ and $w\in A_{\fz}(\rn)$, $L^p_w(\rn)$ is a
ball quasi-Banach space but it may not be a quasi-Banach function space
(see, for instance, \cite[Section 7.1]{SHYY}).
\end{remark}

Using Theorems \ref{thm-B-CZ1}, \ref{thm-B-CZ}, and \ref{dual-T-wT},
we have the following conclusion.

\begin{theorem}\label{thm-wei}
Let $p\in(0,\infty)$ and $w\in A_\infty(\rn)$.
Then both Theorems \ref{thm-B-CZ1} and \ref{thm-B-CZ}
with both $X:=L^p_w(\rn)$ and $p_-:=\frac{p}{q_w}$ hold true;
moreover, Theorem \ref{dual-T-wT} with
$X:=L^p_w(\rn)$, $p_-:=\frac{p}{q_w}$,
$q\in(\max\{1,\frac{p}{q_w},p\},\fz)$, $d\in(0,\min\{1,\frac{p}{q_w}\})$,
and
$$s\in\lf(\frac{n}{d}-n-\dz,\fz\r)
\cap\lf[\lf\lfloor n\lf(\frac{1}{\min\{1,p_-\}}-1\r)\r\rfloor,\infty\r)\cap\zz_+$$
holds true.
\end{theorem}

\begin{proof}
It is easy to show that $L^p_w(\rn)$ is a
ball quasi-Banach space with an absolutely continuous quasi-norm.
Moreover, from a claim in \cite[p.\,28]{CWYZ2020},
we deduce that both Assumptions \ref{assump1} and \ref{assump2}
with $X:=L^p_w(\rn)$, $p_-:=\frac{p}{q_w}$, $r_0\in(d,\frac{p}{q_w})$,
and $p_0\in(\max\{\frac{p}{q_w},p\},q)$ hold true. Thus, all the assumptions of
Theorems \ref{thm-B-CZ1}, \ref{thm-B-CZ}, and \ref{dual-T-wT} with $X:=L^p_w(\rn)$
are satisfied. Then, by Theorems \ref{thm-B-CZ1},
\ref{thm-B-CZ}, and \ref{dual-T-wT} with $X:=L^p_w(\rn)$,
we obtain the desired conclusions, which completes the proof of Theorem \ref{thm-wei}.
\end{proof}

\begin{remark}
\begin{enumerate}
\item[\rm (i)]	
To the best of our knowledge, Theorem \ref{thm-wei} is new.

\item[\rm (ii)]	Moreover, using \cite[Theorem 246]{sdh2020},
we find that, for any $\alpha\in(0,\infty)$ and $q\in[1,\infty]$,
$\mathcal{C}_{\alpha,q,\lfloor\alpha n\rfloor}(\rn)
=\dot{\mathcal{C}}^{\alpha n}(\rn)$ with equivalent norms,
where $\dot{\mathcal{C}}^{\alpha n}(\rn)$ is the homogeneous H\"older--Zygmund
space (see, for instance, \cite[Section 1.6.2]{sdh2020}).
Let $q\in(1,\infty)$, $s\in\zz_+$, $\alpha\in(0,\infty)$, $w\equiv1$, and $X:=L^{\frac{1}{\alpha+1}}_w(\rn)=L^{\frac{1}{\alpha+1}}(\rn)$.
From Remark \ref{comeback}, we infer that, under these assumptions, $\mathcal{L}_{X,q,s}(\rn)$
coincides with the Campanato space $\mathcal{C}_{\alpha,q,s}(\rn)$. Then, by Theorem
\ref{thm-wei}, we conclude that, under these assumptions,
the conclusions of Theorem \ref{thm-B-CZ1}
holds true with $X:=L^{\frac{1}{\alpha+1}}(\rn)$. This further
implies that, under these assumptions, Theorem \ref{thm-B-CZ1} with $X:=L^{\frac{1}{\alpha+1}}(\rn)$
coincides with \cite[Theorem 4.21]{kk2013}.
From these observations, we infer  the following
conclusion: For any $\alpha\in(0,\infty)$,
if letting $\delta\in(\alpha n-\lfloor\alpha n\rfloor,1]$ and $\widetilde{T}$
be the same as in Theorem \ref{thm-B-CZ1}, then $\widetilde{T}$ is bounded on
$\dot{\mathcal{C}}^{\alpha n}(\rn)$ if and only if, for any $\gamma\in\mathbb{Z}^n_+$
with $|\gamma|\leq \lfloor\alpha n\rfloor$, $T^*(x^{\gamma})=0$, which is the known best possible
result on the boundedness of Calder\'on--Zygmund operators on homogeneous H\"older--Zygmund
spaces.
\end{enumerate}
\end{remark}

\subsection{Variable Lebesgue spaces\label{variable}}

In this subsection, we apply
Theorems \ref{thm-B-CZ1}, \ref{thm-B-CZ}, and \ref{dual-T-wT}
to variable Lebesgue spaces.
We now recall the concept of variable Lebesgue spaces
(see, for instance, \cite[p.\,4]{CUF}). Moreover,
we refer the reader to \cite{CUF,dhr2009,ks1991} for more studies on
variable Lebesgue spaces.

\begin{definition}\label{def-pcdot}
Let $p(\cdot):\ \rn\to[0,\infty)$ be a measurable function.
The \emph{variable Lebesgue space $L^{p(\cdot)}(\rn)$} is defined to be
the set of all the measurable functions $f$ on $\rn$ such that
$$
\|f\|_{L^{p(\cdot)}(\rn)}:=\inf\lf\{\lambda\in(0,\infty)
:\ \int_\rn\lf[\frac{|f(x)|}{\lambda}\r]^{p(x)}\,dx\le1\r\}<\infty.
$$
Moreover, $p(\cdot)$ is said to be
\emph{globally log-H\"older continuous} if there
exists a $p_{\infty}\in\rr$ such that, for any $x,y\in\rn$,
$$|p(x)-p(y)|\lesssim\frac{1}{\log(e+1/|x-y|)}
\text{ and }
|p(x)-p_\infty|\lesssim\frac{1}{\log(e+|x|)},$$
where the implicit positive constants are independent of both $x$ and $y$.
\end{definition}

\begin{remark}
Let $p(\cdot):\ \rn\to[0,\infty)$ be a measurable function.
It is easy to show that $L^{p(\cdot)}(\rn)$ is a
ball quasi-Banach space but it may not be a quasi-Banach function space
(see, for instance, \cite[Section 7.1]{SHYY}).
\end{remark}

The following theorem is a corollary of
Theorems \ref{thm-B-CZ1}, \ref{thm-B-CZ}, and \ref{dual-T-wT}.

\begin{theorem}\label{thm-vec}
Let $p(\cdot):\ \rn\to(0,\infty)$ be a globally
log-H\"older continuous function satisfying
$0<\mathop{\mathrm{ess\,inf}}_{x\in\rn}\,p(x)
\le \mathop{\mathrm{ess\,sup}}_{x\in\rn}\,p(x)<\infty$.
Then both Theorems \ref{thm-B-CZ1} and \ref{thm-B-CZ}
with both $X:=L^{p(\cdot)}(\rn)$ and $p_-:=\mathop{\mathrm{ess\,inf}}_{x\in\rn}\,p(x)$ hold true;
moreover, Theorem \ref{dual-T-wT} with
$X:=L^{\vec p}(\rn)$, $p_-:=\mathop{\mathrm{ess\,inf}}_{x\in\rn}\,p(x)$,
$q\in(\max\{1,\mathop{\mathrm{ess\,sup}}_{x\in\rn}\,p(x)\},\fz)$,
$d\in(0,\min\{1,\mathop{\mathrm{ess\,inf}}_{x\in\rn}\,p(x)\})$,
and
$$s\in\lf(\frac{n}{d}-n-\dz,\fz\r)
\cap\lf[\lf\lfloor n\lf(\frac{1}{\min\{1,p_-\}}-1\r)\r\rfloor,\infty\r)\cap\zz_+$$
holds true.
\end{theorem}

\begin{proof}
It is easy to show that $L^{p(\cdot)}(\rn)$ is a
ball quasi-Banach space with an absolutely continuous quasi-norm.
Moreover, using a claim in \cite[p.\,36]{zhyy2022},
we find that both Assumptions \ref{assump1} and \ref{assump2}
with $X:=L^{p(\cdot)}(\rn)$, $p_-:=\mathop{\mathrm{ess\,inf}}_{x\in\rn}\,p(x)$,
$r_0\in(d,\mathop{\mathrm{ess\,inf}}_{x\in\rn}\,p(x))$,
and $p_0\in(\mathop{\mathrm{ess\,sup}}_{x\in\rn}\,p(x),q)$ hold true. Thus, all the assumptions of
Theorems \ref{thm-B-CZ1}, \ref{thm-B-CZ}, and \ref{dual-T-wT} with $X:=L^{p(\cdot)}(\rn)$
are satisfied. Then, from Theorems \ref{thm-B-CZ1}, \ref{thm-B-CZ}, and \ref{dual-T-wT}
with $X:=L^{p(\cdot)}(\rn)$, we deduce the desired conclusions,
which completes the proof of Theorem \ref{thm-vec}.
\end{proof}

\begin{remark}
To the best of our knowledge, Theorem \ref{thm-vec} is new.
\end{remark}

\subsection{Orlicz spaces\label{Orlicz}}

In this subsection, we apply all the main results of this article to Orlicz spaces.
Recall that there exist many operators that are not bounded on Lebesgue spaces
$L^p(\rn)$ with $p\in[1,\fz]$, especially on both $L^1(\rn)$ and $L^{\fz}(\rn)$.
The Orlicz space was introduced to cover the failure
of the boundedness of some integral operators.
We refer the reader to \cite{ans2021,Ho2022,iv1999,s1996} for more studies on this.

A non-decreasing function $\Phi:\ [0,\infty)\to[0,\infty)$
is called an \emph{Orlicz function} if it
satisfies $\Phi(0)= 0$, $\Phi(t)>0$ whenever $t\in(0,\infty)$,
and $\lim_{t\to\infty}\Phi(t)=\infty$.
An Orlicz function $\Phi$ is said to be
of \emph{lower} (resp., \emph{upper}) \emph{type} $p$ with $p\in\rr$ if
there exists a positive constant $C_{(p)}$, depending on $p$,
such that, for any $t\in[0,\infty)$
and $s\in(0,1)$ [resp., $s\in [1,\infty)$],
\begin{align*}
\Phi(st)\le C_{(p)}s^p \Phi(t).
\end{align*}
In addition, an Orlicz function $\Phi$ is said to be of
\emph{positive lower} (resp., \emph{upper}) \emph{type} if it is of lower
(resp., upper) type $p$ for some $p\in(0,\infty)$.

\begin{definition}\label{fine}
Let $\Phi$ be an Orlicz function with lower type
$p_{\Phi}^-\in\rr$ and upper type $p_{\Phi}^+\in\rr$.
The \emph{Orlicz space $L^\Phi(\rn)$} is defined
to be the set of all the measurable functions $f$ on $\rn$ such that
$$\|f\|_{L^\Phi(\rn)}:=\inf\lf\{\lambda\in(0,\infty):\ \int_{\rn}
\Phi\lf(\frac{|f(x)|}{\lambda}\r)\,dx\le1\r\}<\infty.$$
\end{definition}

\begin{remark}
Let $\Phi$ be an Orlicz function on $\rn$ with positive lower type
$p_{\Phi}^-\in(0,\fz)$ and positive upper type $p_{\Phi}^+\in(0,\fz)$.
It has been pointed out in \cite[Section 7.6]{SHYY}
that $L^{\Phi}(\rn)$ is a quasi-Banach
function space and hence a ball quasi-Banach function space.
\end{remark}

Using Theorems \ref{thm-B-CZ1}, \ref{thm-B-CZ}, and \ref{dual-T-wT},
we have the following conclusion.

\begin{theorem}\label{thm-phi}
Let $\Phi$ be an Orlicz function with positive lower type
$p_{\Phi}^-$ and positive upper type $p_{\Phi}^+$ satisfying
$0<p_{\Phi}^-\le p_{\Phi}^+<\infty$.
Then both Theorems \ref{thm-B-CZ1} and \ref{thm-B-CZ}
with both $X:=L^{\Phi}(\rn)$ and $p_-:=p_{\Phi}^-$ hold true; moreover,
Theorem \ref{dual-T-wT} with
$X:=L^{\Phi}(\rn)$, $p_-:=p_{\Phi}^-$,
$q\in(\max\{1,p_{\Phi}^+\},\fz)$, $d\in(0,\min\{1,p_{\Phi}^-\})$,
and
$$s\in\lf(\frac{n}{d}-n-\dz,\fz\r)
\cap\lf[\lf\lfloor n\lf(\frac{1}{\min\{1,p_-\}}-1\r)\r\rfloor,\infty\r)\cap\zz_+$$
holds true.
\end{theorem}

\begin{proof}
Similarly to the proof of \cite[Lemma 4.5]{zyyw2019}, we find that $L^{\Phi}(\rn)$ is a
ball quasi-Banach space with an absolutely continuous quasi-norm.
Moreover, by the claim in \cite[p.\,38]{zhyy2022},
we conclude that both Assumptions \ref{assump1} and \ref{assump2}
with $X:=L^{\Phi}(\rn)$, $p_-:=p_{\Phi}^-$, $r_0\in(d,p_{\Phi}^-)$,
and $p_0\in(p_{\Phi}^+,q)$ hold true. Thus, all the assumptions of
Theorems \ref{thm-B-CZ1}, \ref{thm-B-CZ}, and \ref{dual-T-wT} with $X:=L^{\Phi}(\rn)$
are satisfied. Then, using Theorems \ref{thm-B-CZ1}, \ref{thm-B-CZ}, and \ref{dual-T-wT}
with $X:=L^{\Phi}(\rn)$, we obtain the desired conclusions,
which completes the proof of Theorem \ref{thm-phi}.
\end{proof}

\begin{remark}
To the best of our knowledge, Theorem \ref{thm-phi} is new.
\end{remark}

\subsection{Orlicz-Slice spaces\label{Orlicz-Slice}}

Recall that the Orlicz-slice space was introduced by
Zhang et al. \cite{zyyw2019}, which is a generalization of
the slice spaces proposed by Auscher and Mourgoglou \cite{am2019} and Auscher and
Prisuelos-Arribas \cite{ap2017}. We refer the reader to
\cite{zyw2022} for more studies on
the local Orlicz-slice space.
To be precise, we present the definition of the Orlicz-slice space as follows.

\begin{definition}\label{so}
Let $t,r\in(0,\infty)$ and $\Phi$ be an Orlicz function on $\rn$
with positive lower type $p_{\Phi}^-$ and
positive upper type $p_{\Phi}^+$. The \emph{Orlicz-slice space} $(E_\Phi^r)_t(\rn)$
is defined to be the set of all the measurable functions $f$ on $\rn$ such that
$$
\|f\|_{(E_\Phi^r)_t(\rn)}
:=\lf\{\int_{\rn}\lf[\frac{\|f\mathbf{1}_{B(x,t)}\|_{L^\Phi(\rn)}}
{\|\mathbf{1}_{B(x,t)}\|_{L^\Phi(\rn)}}\r]^r\,dx\r\}^{\frac{1}{r}}<\infty.
$$
\end{definition}

\begin{remark}
Let $t,r\in(0,\infty)$ and $\Phi$ be an Orlicz function on $\rn$
with $0<p_{\Phi}^-\le p_{\Phi}^+<\infty$.
By \cite[Lemma 2.28]{zyyw2019},
we find that $(E_\Phi^r)_t(\rn)$ is a
ball quasi-Banach space, but it may not be a quasi-Banach function space
(see, for instance, \cite[Remark 7.43(i)]{zyyw}).
\end{remark}

The following theorem is a corollary of
Theorems \ref{thm-B-CZ1}, \ref{thm-B-CZ}, and \ref{dual-T-wT}.

\begin{theorem}\label{thm-Et}
Let $t,r\in(0,\infty)$ and $\Phi$ be an Orlicz function
with positive lower type $p_{\Phi}^-$ and
positive upper type $p_{\Phi}^+$ satisfying
$0<p_{\Phi}^-\le p_{\Phi}^+<\infty$.
Then Theorems \ref{thm-B-CZ1} and \ref{thm-B-CZ}
with $X:=(E_\Phi^r)_t(\rn)$ and $p_-:=\min\{r,p_{\Phi}^-\}$ hold true; moreover,
Theorem \ref{dual-T-wT} with
$X:=(E_\Phi^r)_t(\rn)$, $p_-:=\min\{r,p_{\Phi}^-\}$,
$q\in(\max\{1,r,p_{\Phi}^+\},\fz)$, $d\in(0,\min\{1,r,p_{\Phi}^-\})$,
and
$$s\in\lf(\frac{n}{d}-n-\dz,\fz\r)
\cap\lf[\lf\lfloor n\lf(\frac{1}{\min\{1,p_-\}}-1\r)\r\rfloor,\infty\r)\cap\zz_+$$
holds true.
\end{theorem}

\begin{proof}
Similarly to the proof of \cite[Lemma 4.5]{zyyw2019}, we find that
$L^{\Phi}(\rn)$ is a ball quasi-Banach space with an absolutely continuous quasi-norm.
Moreover, by a claim in \cite[p.\,39]{zhyy2022},
we conclude that both Assumptions \ref{assump1} and \ref{assump2}
with $X:=(E_\Phi^r)_t(\rn)$,
$p_-:=\min\{r,p_{\Phi}^-\}$, $r_0\in(d,\min\{r,p_{\Phi}^-\})$,
and $p_0\in(\max\{r,p_{\Phi}^+\},q)$ hold true. Thus, all the assumptions of
Theorems \ref{thm-B-CZ1}, \ref{thm-B-CZ}, and \ref{dual-T-wT} with $X:=(E_\Phi^r)_t(\rn)$
are satisfied. Then, using Theorems \ref{thm-B-CZ1}, \ref{thm-B-CZ},
and \ref{dual-T-wT} with $X:=(E_\Phi^r)_t(\rn)$,
we obtain the desired conclusions, which completes the proof of Theorem \ref{thm-Et}.
\end{proof}

\begin{remark}
To the best of our knowledge, Theorem \ref{thm-Et} is new.
\end{remark}

\subsection{Morrey spaces\label{Morrey}}

Recall that Morrey \cite{MCB1938} introduced the Morrey space
$M_{r}^{p}(\mathbb{R}^{n})$ with $0<r \leq p<\infty$
to study the regularity
of solutions to certain equations.
The Morrey space has many applications in the theory of elliptic partial
differential equations, potential theory, and harmonic analysis
(see, for instance, \cite{cf1987,sdh2020,tyy-2019}).

\begin{definition}\label{Def-Morrey}
Let $0<r \leq p<\infty$. The \emph{Morrey space}
$M_{r}^{p}(\mathbb{R}^{n})$ is
defined to be the set of all the measurable
functions $f$ on $\rn$ such that
$$
\|f\|_{M_{r}^{p}(\mathbb{R}^{n})}:=\sup
_{B \in \mathbb{B}(\rn)}|B|^{\frac{1}{p}-\frac{1}{r}}\|f\|_{L^{r}(B)}<\infty.
$$
\end{definition}

\begin{remark}\label{Rem-Morrey}
Let $0<r \leq p<\infty$. By Definition \ref{Def-Morrey}, we easily find that
$M_{r}^{p}(\rn)$ is a ball quasi-Banach function space.
However, it has been pointed out in \cite[p.\,87]{SHYY} that $M_{r}^{p}(\rn)$ may not be a
quasi-Banach function space. Moreover, $M_r^p(\rn)$ does not have an absolutely continuous quasi-norm unless $r=p$, namely, $M^{p}_{r}(\rn)=L^r(\rn)$.
\end{remark}

The following theorem is a corollary of both
Theorems \ref{thm-B-CZ1} and \ref{thm-B-CZ}.

\begin{theorem}\label{ap-M}
Let $0<r \leq p<\frac{n}{\alpha}$. Then
Theorems \ref{thm-B-CZ1} and \ref{thm-B-CZ}
with both $X:=M_{r}^{p}(\rn)$ and $p_-:=r$ hold true.
\end{theorem}

\begin{proof}
By Definition \ref{Def-Morrey}, we easily find that
$M_{r}^{p}(\rn)$ is a ball quasi-Banach function space.
In addition, by \cite[Lemma 2.5]{tx2005} (see also \cite[Lemma 7.2]{zyyw}),
we conclude that
$M_r^p(\rn)$ satisfies Assumption
\ref{assump1} with both $X:=M_r^p(\rn)$ and $p_-:=r$.
Thus, all the assumptions of
Theorems \ref{thm-B-CZ1} and \ref{thm-B-CZ} with $X:=M_r^p(\rn)$ are satisfied.
Then, using Theorems \ref{thm-B-CZ1} and \ref{thm-B-CZ}
with $X:=M_r^p(\rn)$, we obtain the desired conclusions,
which completes the proof of
Theorem \ref{ap-M}.
\end{proof}

\begin{remark}\label{rem-Morrey}
\begin{enumerate}
\item[\rm (i)]	
Let $0<r\leq p<\infty$.
As Remark \ref{Rem-Morrey} point out that the Morrey space $M^{p}_{r}(\rn)$ does not have an absolutely continuous quasi-norm unless $r=p$,
Theorem \ref{dual-T-wT} can not be applied
to  $M^{p}_{r}(\rn)$.

\item[\rm (ii)]	
To the best of our knowledge, Theorem \ref{ap-M} is new.

\item[\rm (iii)]
It is worth pointing out that there exist many studies about
important operators on Morrey-type
space; we refer the reader to
\cite{DGNSS,HSS2016,ho2017,SS2017,SST2009,Ho2021}.
\end{enumerate}
\end{remark}

\subsection{Mixed-norm Lebesgue spaces\label{Mixed-Norm}}

The mixed-norm Lebesgue space
$L^{\vec{p}}(\mathbb{R}^{n})$
was studied by Benedek and Panzone
\cite{BAP1961} in 1961, which can be
traced back to H\"ormander \cite{HL1960}.
We refer the reader to
\cite{CGN2017,CGG2017,CGN20172,GN2016,HLY2019,HLYY2019,HYacc}
for more studies on mixed-norm type spaces.

\begin{definition}\label{mixed}
Let $\vec{p}:=(p_{1}, \ldots, p_{n})
\in(0, \infty]^{n}$. The \emph{mixed-norm Lebesgue space
$L^{\vec{p}}(\mathbb{R}^{n})$}
is defined to be the set of all the measurable
functions $f$ on $\rn$ such that
\begin{align*}
\|f\|_{L^{\vec{p}}(\mathbb{R}^{n})}:=\left\{\int_{\mathbb{R}}
\cdots\left[\int_{\mathbb{R}}\left|f(x_{1}, \ldots,
x_{n})\right|^{p_{1}} \,d x_{1}\right]^{\frac{p_{2}}{p_{1}}}
\cdots \,dx_{n}\right\}^{\frac{1}{p_{n}}}<\infty
\end{align*}
with the usual modifications made when $p_i=\infty$ for some $i\in\{1,\ldots,n\}$.
Moreover, let
\begin{align}\label{p-p+}
p_{-}:=\min\{p_{1},\ldots, p_{n}\}\text{ and }p_{+}:=\max\{p_{1},\ldots, p_{n}\}.
\end{align}
\end{definition}

\begin{remark}\label{mix-r}
Let $\vec{p}\in(0, \infty)^{n}$.	
From Definition \ref{mixed}, we easily deduce that
$L^{\vec{p}}(\mathbb{R}^{n})$
is a ball quasi-Banach space.
However, as was pointed out in \cite[Remark 7.21]{zyyw}, $L^{\vec{p}}(\rn)$
may not be a quasi-Banach function space.
\end{remark}

The following theorem is a corollary of
Theorems \ref{thm-B-CZ1}, \ref{thm-B-CZ}, and \ref{dual-T-wT}.

\begin{theorem}\label{apply3}
Let $\vec{p}:=(p_{1}, \ldots, p_{n})
\in(0, \infty)^{n}$. Then Theorems \ref{thm-B-CZ1}, \ref{thm-B-CZ},
and \ref{dual-T-wT} with both
$X:=L^{\vec{p}}(\mathbb{R}^{n})$ and $p_{-}:=\min\{p_{1},\ldots, p_{n}\}$ hold true.
\end{theorem}

\begin{proof}
Let all the symbols be the same as in the
present theorem, and let $p_+$ be the same as in \eqref{p-p+}.
By Definition \ref{mixed}, we easily find that
$L^{\vec{p}}(\mathbb{R}^{n})$
is a ball quasi-Banach space with an absolutely continuous quasi-norm.
In addition, from \cite[Lemma 3.7]{HLY2019}, we deduce that Assumption
\ref{assump1} with both $X:=L^{\vec{p}}(\rn)$ and
$p_-$ in \eqref{p-p+} holds true.
In addition, using both the dual theorem of $L^q(\rn)$
(see \cite[p.\,304, Theorem 1.a]{BAP1961}) and \cite[Lemma 3.5]{HLY2019}
(see also \cite[Lemma 4.3]{HYacc}), we conclude
that Assumption \ref{assump2} with
any $r_0\in(0,p_-)$ and $p_0\in(p_+,\fz)$ also holds true.
Thus, all the assumptions of Theorems \ref{thm-B-CZ1}, \ref{thm-B-CZ}, and \ref{dual-T-wT}
with $X:=L^{\vec{p}}(\rn)$ are satisfied.
Then, using Theorems \ref{thm-B-CZ1}, \ref{thm-B-CZ}, and \ref{dual-T-wT} with $X:=L^{\vec{p}}(\rn)$,
we obtain the desired conclusions,
which completes the proof
of Theorem \ref{apply3}.
\end{proof}

\begin{remark}
To the best of our knowledge, Theorem \ref{apply3} is new.
\end{remark}

\subsection{Local generalized Herz spaces\label{H-H}}

The local generalized Herz space was originally introduced by Rafeiro and Samko \cite{RS2020},
which is the generalization of the classical
homogeneous Herz space and connects with the generalized Morrey type space.
Very recently, Li et al. \cite{LYH2022} developed a complete
real-variable theory of Hardy spaces
associated with the local generalized Herz space.
We now present the concepts of both the function class
$M\left(\mathbb{R}_{+}\right)$ and the
local generalized Herz space
$\dot{\mathcal{K}}_{\omega, \mathbf{0}}^{p,
q}(\mathbb{R}^{n})$
(see \cite[Definitions 2.1 and 2.2]{RS2020}
and also \cite[Definitions 1.1.1 and 1.2.1]{LYH2022}).

\begin{definition}
Let $\mathbb{R}_+:=(0,\infty)$. The \emph{function class}
$M\left(\mathbb{R}_{+}\right)$ is defined to
be the set of all the positive functions
$\omega$ on $\mathbb{R}_{+}$ such that, for
any $0<\delta<N<\infty$,
$$
0<\inf _{t \in(\delta, N)} \omega(t) \leq
\sup _{t \in(\delta, N)} \omega(t)<\infty
$$
and there exist four constants $\alpha_{0}$,
$\beta_{0}$, $\alpha_{\infty}$,
$\beta_{\infty} \in \mathbb{R}$ such that
\begin{itemize}
\item[\rm (i)] for any $t \in(0,1]$, $\omega(t)
t^{-\alpha_{0}}$ is almost increasing and
$\omega(t) t^{-\beta_{0}}$ is almost
decreasing;

\item[\rm (ii)] for any $t \in[1, \infty)$, $\omega(t)
t^{-\alpha_{\infty}}$ is almost increasing
and $\omega(t) t^{-\beta_{\infty}}$ is
almost decreasing.
\end{itemize}
\end{definition}

\begin{definition}\label{def-Lw}
Let $p,q \in(0, \infty)$ and $\omega \in
M\left(\mathbb{R}_{+}\right)$.
The \emph{local generalized Herz space}
$\dot{\mathcal{K}}_{\omega, \mathbf{0}}^{p,
q}(\mathbb{R}^{n})$ is defined to
be the set of all the measurable functions
$f$ on $\mathbb{R}^{n}$ such that
$$
\|f\|_{\dot{\mathcal{K}}_{\omega, \mathbf{0}}^{p,
q}\left(\mathbb{R}^{n}\right)}:=\left\{\sum_{k \in
\mathbb{Z}}\left[\omega\left(2^{k}\right)\right]^{q}\left\|f
\mathbf{1}_{B\left(\mathbf{0}, 2^{k}\right) \setminus
B\left(\mathbf{0},
2^{k-1}\right)}\right\|_{L^{p}\left(\mathbb{R}^{n}\right)}
^{q}\right\}^{\frac{1}{q}}
$$
is finite.
\end{definition}

\begin{remark}
Let all the symbols be the same as in Definition \ref{def-Lw}.
It has been proved in \cite[Theorem 1.2.42]{LYH2022} that
$\dot{\mathcal{K}}_{\omega,\mathbf{0}}^{p,
r}(\mathbb{R}^{n})$ is a ball quasi-Banach space.	
However, as was pointed out in \cite[Remark 4.15]{cjy-01},
$\dot{\mathcal{K}}_{\omega, \mathbf{0}}^{p,q}(\mathbb{R}^{n})$
may not be a quasi-Banach function space.
\end{remark}

\begin{definition}
Let $\omega$ be a positive
function on $\mathbb{R}_{+}$. Then the
\emph{Matuszewska-Orlicz indices} $m_{0}(\omega)$,
$M_{0}(\omega)$, $m_{\infty}(\omega)$, and
$M_{\infty}(\omega)$ of $\omega$ are
defined, respectively, by setting, for any
$h \in(0, \infty)$,
$$m_{0}(\omega):=\sup _{t \in(0,1)}
\frac{\ln (\varlimsup\limits_{h \to
0^{+}} \frac{\omega(h
t)}{\omega(h)})}{\ln t},\
M_{0}(\omega):=\inf _{t \in(0,1)}
\frac{\ln (
\varliminf\limits_{h\to0^{+}}
\frac{\omega(h t)}{\omega(h)})}{\ln
t},$$
$$m_{\infty}(\omega):=\sup _{t \in(1,
\infty)} \frac{\ln (\varliminf
\limits_{h \to\infty}
\frac{\omega(h t)}{\omega(h)})}{\ln
t},$$
and
$$M_{\infty}(\omega):=\inf _{t \in(1, \infty)}
\frac{\ln (\varlimsup\limits_{h
\to \infty} \frac{\omega(h
t)}{\omega(h)})}{\ln t}.
$$
\end{definition}

The following theorem is a corollary of
Theorems \ref{thm-B-CZ1}, \ref{thm-B-CZ}, and \ref{dual-T-wT}.

\begin{theorem}\label{apply5}
Let $p,r\in(0, \infty)$ and $\omega \in
M\left(\mathbb{R}_{+}\right)$. Then Theorems \ref{thm-B-CZ1},
\ref{thm-B-CZ}, and \ref{dual-T-wT}
with both $X:=\dot{\mathcal{K}}_{\omega,
\mathbf{0}}^{p,r}(\mathbb{R}^{n})$ and
$$
p_-:=\min\lf\{p, \frac{n}{\max\lf\{M_{0}(\omega),M_{\infty}(\omega)\r\}+n/p}\r\}
$$
hold true.
\end{theorem}

\begin{proof}
Let
$$
r_0 \in\lf(0, \min \lf\{1,p_-, r\r\}\r)
$$
and
$$
p_0 \in\lf(\max\lf\{p, \frac{n}{\min \{m_{0}(\omega),m_{\infty}(\omega)\}+n / p}\r\}, \infty\r].
$$
From \cite[Theorems 1.2.42 and 1.4.1]{LYH2022}, we deduce that
$\dot{\mathcal{K}}_{\omega,\mathbf{0}}^{p,
r}(\mathbb{R}^{n})$ is a ball quasi-Banach space with an absolutely continuous quasi-norm.
To prove the required conclusions,
it suffices to show that
$\dot{\mathcal{K}}_{\omega,\mathbf{0}}^{p,r}(\rn)$ satisfies
all the assumptions of Theorems \ref{thm-B-CZ1}, \ref{thm-B-CZ}, and \ref{dual-T-wT}
with $X:=\dot{\mathcal{K}}_{\omega,\mathbf{0}}^{p,r}(\rn)$.
Indeed, by \cite[Lemma 4.3.10]{LYH2022}, we conclude that Assumption \ref{assump1} with
$X:=\dot{\mathcal{K}}_{\omega,\mathbf{0}}^{p,r}(\rn)$ holds true.
Moreover, from \cite[Lemma 1.8.6]{LYH2022}, we deduce that Assumption \ref{assump2} with $X:=\dot{\mathcal{K}}_{\omega,\mathbf{0}}^{p,r}(\mathbb{R}^{n})$ also holds true.
Thus, all the assumptions of Theorems \ref{thm-B-CZ1}, \ref{thm-B-CZ}, and \ref{dual-T-wT}
with $X:=\dot{\mathcal{K}}_{\omega,\mathbf{0}}^{p,r}(\rn)$ are satisfied.
Then, using Theorems \ref{thm-B-CZ1}, \ref{thm-B-CZ}, and \ref{dual-T-wT}
with $X:=\dot{\mathcal{K}}_{\omega,\mathbf{0}}^{p,r}(\rn)$,
we obtain the desired conclusions,
which completes the proof
of Theorem \ref{apply5}.
\end{proof}

\begin{remark}
To the best of our knowledge, Theorem
\ref{apply5} is new.
\end{remark}

\subsection{Mixed Herz spaces\label{M-H}}

We now present the following definition of mixed Herz spaces, which is just
\cite[Definition 2.3]{zyz2022}.

\begin{definition}\label{mhz}
Let $\vec{p}:=(p_{1},\ldots,p_{n}),\vec{q}:=(q_{1},\ldots,q_{n})
\in(0,\infty]^{n}$, $\vec{\alpha}:=
(\alpha_{1},\ldots,
\alpha_{n})\in\rn$, and $R_{k_i}:=(-2^{k_i},2^{k_i})\setminus
(-2^{k_i-1},2^{k_i-1})$ for any
$k_i\in\zz$ and $i\in\{1,\ldots,n\}$.
The \emph{mixed Herz space}
$\dot{E}^{\vec{\alpha},\vec{p}}_{\vec{q}}(\rn)$ is
defined to be the
set of all the functions
$f\in \mathscr{M}(\rn)$ such
that
\begin{align*}
\|f\|_{\dot{E}^{\vec{\alpha},\vec{p}}_{\vec{q}}
(\rn)}:&=\lf\{\sum_{k_{n} \in
\zz}2^{k_{n}
p_{n}\alpha_{n}}
\lf[\int_{R_{k_{n}}}\cdots\lf\{\sum_{k_{1}
\in \zz}
2^{k_{1}p_{1}\alpha_{1}}\r.\r.\r.\\
&\lf.\lf.\lf.\quad\times\lf[\int_{R_{k_{1}}}|f(x_{1},
\ldots,x_{n})|^{q_{1}}\,dx_{1} \r]^{\f{p_{1}}{q_{1}}}
\r\}^{\f{q_{2}}{p_{1}}}\cdots\,dx_{n}\r]^{\f{p_{n}}{q_{n}}}\r\}
^{\f{1}{p_n}}\\
&=:\,\lf\|\cdots\|f\|_{\dot{K}^{\alpha_{1},p_{1}}_{q_{1}}
(\rr)}\cdots\r\|_{\dot{K}^{\alpha_{n},p_{n}}_{q_{n}}(\rr)}
<\infty
\end{align*}
with the usual modifications made when $p_{i}
=\infty$ or
$q_{j}=\infty$ for some $i,j\in\{1,\ldots,n \}$,
where $\|\cdots\|f\|_{\dot{K}^{\alpha_{1},p_{1}}_{q_{1}}(\rr)}
\cdots\|_{\dot{K}^{\alpha_{n},p_{n}}_{q_{n}}(\rr)}$ denotes the
norm obtained after taking successively
the $\dot{K}^{\alpha_{1},p_{1}}_{q_{1}}(\rr)$-norm
to $x_{1}$, $\ldots$, the $\dot{K}^{\alpha_{n-1},p_{n-1}}_{q_{n-1}}
(\rr)$-norm to $x_{n-1}$,
and the $\dot{K}^{\alpha_{n},p_{n}}_{q_{n}}(\rr)$-norm to $x_{n}$.
\end{definition}

Recall that, to study the Lebesgue points of functions in mixed-norm Lebesgue spaces,
Huang et al. \cite{HWYY2021} introduced a special case of the mixed Herz space and, later, Zhao
et al. \cite{zyz2022} generalized it to the above case.
Also, both the dual theorem and the Riesz--Thorin
interpolation theorem on the above mixed Herz space have been fully studied in \cite{zyz2022}.

\begin{remark}
Let $\vec{p}:=(p_{1},\ldots,p_{n}),\vec{q}:=(q_{1},\ldots,q_{n})
\in(0,\infty]^{n}$ and $\vec{\alpha}:=(\alpha_{1},\ldots,\alpha_{n})\in\rn$.
By \cite[Proposition 2.22]{zyz2022}, we conclude that
$\dot{E}^{\vec{\alpha},\vec{p}}_{\vec{q}}(\rn)$
is a ball quasi-Banach space if and only if,
for any $i\in\{1,\ldots,n\}$, $\alpha_i\in(-\frac{1}{q_i},\fz)$.
However, as was pointed out in \cite[Remark 4.22]{cjy-01},
$\dot{E}^{\vec{\alpha},\vec{p}}_{\vec{q}}(\rn)$
may not be a quasi-Banach function space.
\end{remark}

The following theorem is a corollary of
Theorems \ref{thm-B-CZ1},
\ref{thm-B-CZ}, and \ref{dual-T-wT}.

\begin{theorem}\label{apply7}
Let $\vec{p}:=(p_{1},\ldots,p_{n}),\vec{q}:=(q_{1},\ldots,q_{n})
\in(0,\infty]^{n}$ and $\vec{\alpha}:=
(\alpha_{1},\ldots,
\alpha_{n})\in\rn$. Then Theorems \ref{thm-B-CZ1},
\ref{thm-B-CZ}, and \ref{dual-T-wT}
with both $X:=\dot{E}^{\vec{\alpha},\vec{p}}_{\vec{q}}(\rn)$
and
\begin{align}\label{4.2x}
p_-:&=\min\lf\{p_1,\ldots,p_n,
q_1,\ldots,q_n,\phantom{\lf(\alpha_{1}+\frac{1}{q_{1}}\r)^{-1}}\r.\noz\\
&\quad\qquad\lf.\lf(\alpha_{1}+\frac{1}{q_{1}}\r)^{-1},
\ldots,\lf(\alpha_{n}+\frac{1}{q_{n}}\r)^{-1} \r \}
\end{align}
hold true.
\end{theorem}

\begin{proof}
Let
$$r_0\in\lf(0,\min\lf\{1,p_-\r\}\r)$$
and
$$p_0\in\lf(\max\lf\{p_1,\ldots,p_n,
q_1,\ldots,q_n,\lf(\alpha_{1}
+\frac{1}{q_{1}}\r)^{-1}
,\ldots,\lf(\alpha_{n}+\frac
{1}{q_{n}}\r)^{-1}\r\},\infty \r).$$
Using \cite[Propositions 2.8 and 2.22]{zyz2022}, we conclude that
$\dot{E}^{\vec{\alpha},\vec{p}}_{\vec{q}}(\rn)$
is a ball quasi-Banach space with an absolutely continuous quasi-norm.
By both \cite[Lemma 5.3(i)]{zyz2022} and its proof, we conclude that
Assumption \ref{assump1} with both
$X:=\dot{E}^{\vec{\alpha},\vec{p}}_{\vec{q}}(\rn)$
and $p_-$ in \eqref{4.2x} holds true.
In addition, using both \cite[Lemma 5.3(ii)]{zyz2022}
and its proof, we find that Assumption \ref{assump2}
with $X:=\dot{E}^{\vec{\alpha},\vec{p}}_{\vec{q}}(\rn)$ also holds true.
Thus, all the assumptions of
Theorems \ref{thm-B-CZ1},
\ref{thm-B-CZ}, and \ref{dual-T-wT}
with $X:=\dot{E}^{\vec{\alpha},\vec{p}}_{\vec{q}}(\rn)$ are satisfied.
Then, using Theorems \ref{thm-B-CZ1},
\ref{thm-B-CZ}, and \ref{dual-T-wT}
with $X:=\dot{E}^{\vec{\alpha},\vec{p}}_{\vec{q}}(\rn)$,
we obtain the desired conclusions,
which completes the proof
of Theorem \ref{apply7}.
\end{proof}

\begin{remark}
To the best of our knowledge, Theorem \ref{apply7} is new.
\end{remark}

\noindent\textbf{Acknowledgements}\quad
The authors would like to thank both referees for their very carefully 
reading and several valuable remarks which definitely 
improve the quality of this article.

\medskip

\noindent\textbf{Data Availability Statement}\quad Data
sharing not applicable to
this article as no datasets were generated or
analysed during the current study.

\medskip

\noindent\textbf{Compliance with ethical standards}

\medskip

\noindent\textbf{Conflict of interest}\quad The author declares that they have no conflict interests.

\bigskip

\noindent Yiqun Chen, Hongchao Jia and Dachun Yang (Corresponding
author)

\smallskip

\noindent  Laboratory of Mathematics and Complex Systems
(Ministry of Education of China),
School of Mathematical Sciences, Beijing Normal University,
Beijing 100875, The People's Republic of China

\smallskip

\noindent {\it E-mails}: \texttt{yiqunchen@mail.bnu.edu.cn} (Y. Chen)

\noindent\phantom{{\it E-mails:}} \texttt{hcjia@mail.bnu.edu.cn} (H. Jia)

\noindent\phantom{{\it E-mails:}} \texttt{dcyang@bnu.edu.cn} (D. Yang)

\end{document}